\documentclass{article}
\usepackage[utf8]{inputenc}
\usepackage[english]{babel}
\usepackage{comment}
\usepackage[toc,page]{appendix}
\usepackage{graphicx}
\usepackage{amsmath,amsthm,amsfonts,amssymb}
\usepackage{amsmath}
\usepackage{enumitem}
\usepackage{hyperref}
\usepackage{tikz,tikz-cd}
\newtheorem{thm}{Theorem}
\newtheorem{proposition}{Proposition}[section]

\newtheorem{rem}[proposition]{Remark}
\newtheorem{lemme}[proposition]{Lemma}

\newcommand{\p}{\partial}
\newcommand{\Ra}{\mathbb{R}}
\numberwithin{equation}{section}
\title{Transonic limit of traveling waves of the Euler-Korteweg system}
\author{
Marc-Antoine Vassenet \thanks{Universit\'e Paris Dauphine, PSL Research University, Ceremade, Umr Cnrs 7534, Place du Mar\' echal De Lattre De Tassigny 75775 Paris cedex 16 (France).}}
\begin{document}
\maketitle
\begin{abstract}
We prove the convergence in the  transonic limit of two-dimensional traveling waves of the E-K system, up to rescaling, toward a ground state of the Kadomtsev-Petviashvili Equation. Similarly, in dimension one we prove the convergence in the transonic limit of solitons toward the soliton of the Korteweg de Vries equation.
\end{abstract}
\tableofcontents
\section{Introduction}
The Euler Korteweg system, in dimension $d$, reads
\begin{equation}
\label{E-K}
\tag{E-K}
  \left\{
      \begin{aligned}
        &\partial_t \rho + \text{div}(\rho u)=0, \\
        &\partial_tu+u. \nabla u+ \nabla g(\rho)=\nabla \bigg(K(\rho) \Delta\rho +\frac{1}{2}K'(\rho)| \nabla \rho|^2 \bigg), \quad t\in \Ra, \, x\in \mathbb{R}^d,
      \end{aligned}
    \right.
\end{equation}
where $\rho>0$ is the density of the fluid, $u \in \mathbb{R}^d$ is its velocity field, the right hand side of the second
line is the  capillary tensor. The functions $K$, $g$ are defined on  $\mathbb{R}^{+}_*$ and smooth, the function $K$ is positive. 
When the velocity is irrotational  i.e $u = \nabla \phi$ for some $\phi$ that cancels at infinity, the momentum equation rewrites
$$
\partial_t \phi +\frac{|\nabla \phi|^2}{2}+g(\rho)=K(\rho) \Delta \rho + \frac{1}{2}K'(\rho)|\nabla \rho|^2.
$$
There is a formally conserved energy
\begin{equation}
\label{energie E K}
    E(\rho, \phi) = \int_{\Ra^d} \frac{K(\rho)|\nabla \rho|^2+\rho |\nabla \phi|^2}{2}+G(\rho)dx,
\end{equation}
where $G$ is a primitive of $g$ with a later specified integration constant. Moreover we have a momentum 
\begin{equation}
\label{moment}
    P(\rho, \phi)= \int_{\mathbb{R}^d}(\rho-\rho_0)\partial_1 \phi dx.
\end{equation}
which makes sense when $\rho-\rho_0 \in L^2,\, \nabla \phi \in L^2$.
 The energy makes sense for $(\rho,u)$ localized near the constant state $(\rho_0,0)$, which will be our framework. We call traveling wave a solution of \eqref{E-K} of the form
$$
(\rho(x.\mathrm{n}-ct, x_\perp ), u(x\cdot \mathrm{n}-ct, x_\perp )), \; x_\perp=(x_1,x_2,...,x_d)-x\cdot \mathrm{n},
$$
where $c$ is the speed of propagation and   $\mathrm{n}$ the
direction of the speed.
The
direction of the speed does not matter, thus we let $\mathrm{n}=\mathrm{e_1}$.
A traveling wave solves
\begin{equation}
\label{E-K sim}
  \left\{
      \begin{aligned}
        &-c \partial_1 \rho+\text{div}(\rho \nabla \phi)=0, \\
        & -c \partial_1 \phi+g(\rho)-K(\rho)\Delta(\rho) +\frac{|\nabla\phi|^2}{2}-K'(\rho)\frac{|\nabla \rho|^2}{2}=0.
      \end{aligned}
    \right.
\end{equation}
In \cite{audiard2017small}, Audiard  proves the existence of  traveling waves in dimension two, localized near the constant state $(\rho_0, 0)$
with $g(\rho_0) = 0,\, g'(\rho_0)>0
$. Their speed is close but less than the speed of sound that we define now. When neglecting the capilary tensor and linearizing this system near $(\rho_0,0)$ i.e.  $\rho=\rho_0+r$,  we obtain the Euler equation
\begin{equation}
  \left\{
      \begin{aligned}
        &\partial_t r +\rho_0 \text{div}(u)=0, \\
        &\partial_tu+g'(\rho_0)\nabla r=0.
      \end{aligned}
    \right.
\end{equation}
The speed of sound is
$$
c_s=\sqrt{\rho_0 g'(\rho_0)}.
$$
For simplification, we use the following rescaling 
\begin{align*}
    &(\rho,\phi)=\bigg(\rho_0\rho_r\bigg(\sqrt{\frac{g'(\rho_0)}{\rho_0}x}\bigg), \phi_r\bigg(\sqrt{\frac{g'(\rho_0)}{\rho_0}x}\bigg)\bigg), \, K_r(\rho_r)=\frac{K(\rho_0 \rho_r)}{\rho_0}, \\&
    g_r(\rho_r)=\frac{g(\rho_0 \rho_r)}{g'(\rho_0)\rho_0}, c_r=\frac{c}{\sqrt{\rho_0g'(\rho_0)}}.
\end{align*}
Then \eqref{E-K sim} becomes
\begin{equation}
\label{E-K simple}
  \left\{
      \begin{aligned}
        &-c_r \partial_1 \rho_r+\text{div}(\rho_r \nabla \phi_r)=0, \\
        & -c_r \partial_1 \phi_r+g_r(\rho_r)-K_r(\rho_r)\Delta(\rho_r) +\frac{|\nabla\phi_r|^2}{2}-K'(\rho_r)\frac{|\nabla \rho_r|^2}{2}=0.
      \end{aligned}
    \right.
\end{equation}
In this system the constant state is 1,
 $g_r(1)=0$, $g_r'(1)=1$ and the speed of sound is $\sqrt{1g'(1)}=1$. We will forget the subscript $r$ and focus on the rescaled system, i.e. we will assume through the rest of the paper that 
 \begin{equation}
 \label{hypothese g}
     g(1)=0,\;g'(1)=1\; \text{and}\; c_s=1.
 \end{equation}
 Solutions of \eqref{E-K simple} with speed near the speed of sound are known to exist, the precise existence statement of \cite{audiard2017small} is the following:
\begin{thm}(\cite{audiard2017small} Theorem 1.1, proposition 3.3 and proposition 2.3)
\label{existencethm}
In dimension two, under the assumption $\Gamma:=3+g''(1) \neq 0 $, \, there exists $p_0>0$ such that for any $0\leq p \leq p_0$ we have $(\rho_p,\phi_p) \in \cap_{j\geq 0}\big(1+H^j\big) \times \dot{H}^{j+1} $, solution of $(\ref{E-K simple})$ for some  $c_p>0$, with $P(\rho_p,\phi_p)=p$.\newline  Moreover there exists $\alpha, \beta,C >0$ such that, for any $0\leq p\leq p_0 $
\begin{equation}
\label{energie-moment}
      p-\beta p^3 \leq E(\rho_p,\phi_p)\leq p- \alpha p^3, 
\end{equation}
\begin{equation}
1-\beta p^2 \leq c_p \leq 1-\alpha p^2,
\end{equation}
\begin{equation}
\label{2.3}
    ||\rho_p-1||_{\infty}\leq C \sqrt{E(\rho_p,\phi_p)} \quad  {\rm for} \; p \ll 1.
\end{equation}
\begin{equation}
\label{proposition 3.3}    
||\rho_p-1||_{\infty}\geq C p^2 \quad {\rm for} \; p \ll 1.
\end{equation}
\begin{rem}
As has been proven for the Schrödinger equation (see \cite{chiron2017travelingequationnonzero}), it is possible that such solitons exist in higher dimensions.
\end{rem}
\end{thm}
 To construct the traveling waves of theorem \ref{existencethm} the author in  \cite{audiard2017small} solves a minimization problem. On the space  $\mathcal{H}=\left\{(\rho,\phi) \in (1+H^1) \times \dot{H^1} \right\}$ the momentum is well-defined, however the energy \eqref{energie E K} does not make sense. For example the term $\rho |\nabla \phi|^2$ is not necessarily integrable.  The solution is to work with a modified energy $\widetilde{E}$ which has nice coercive properties and such that $\widetilde{E}= E$ if $| \rho -1| \ll 1$. Then the author finds $(\rho, \phi)$ solution of the minimization problem
\begin{equation}
\label{minimisateur pour E-k}
\inf \left\{  \widetilde{E}(\rho,\phi), \, (\rho,\phi) \in \mathcal{H}: P(\rho, \phi)=p \right\},
\end{equation}
for small $p$, such that the minimizer is smooth and satisfies $| \rho-1 | \ll 1$.
\newline Our aim is to describe the
asymptotic behaviour, as $p\rightarrow 0$, of the traveling waves $(\rho_p,\phi_p)$. It is instructive to compare our problem to the extensive litterature studying the nonlinear Schrödinger equation.
 Indeed in the case $K=\kappa / \rho$, with $\kappa$ a positive constant, up to a rescaling there exists a formal correspondance with the nonlinear Schrödinger equation
\begin{equation}
    \label{NLS}
    \tag{NLS}
    i\partial_t \Psi+\Delta \Psi =g\big(|\Psi|^2)\Psi,
\end{equation}
using the \textit{Madelung transform}\, $ \big( \rho, \nabla \phi\big) \mapsto \Psi :=\sqrt{\rho}e^{i\phi}$ (see \cite{CDS} for more details).
In the case $g(\rho)= \rho-1 $  \eqref{NLS} is called the Gross-Pitaevskii equation
\begin{equation}
\label{GP}
\tag{GP}
i\partial_t \Psi+ \Delta \Psi =\Psi ( |\Psi|^2-1) \; \text{on}  \; \mathbb{R}^d \times \mathbb{R}.
\end{equation}
The counterpart of \eqref{energie E K} is 
\begin{equation}
    E(\Psi)= \int_{\mathbb{R}^d} |\nabla \Psi|^2+\frac{1}{2}  \int_{\mathbb{R}^d} \big(1 -|\Psi|^2\big)^2,
\end{equation}
and that of \eqref{moment} is
\begin{equation}
    P(\Psi)=\frac{1}{2} \int_{\mathbb{R}^d} \mathfrak{Re}\big( i \nabla \Psi( \overline{\Psi-1})\big) .
\end{equation}
We also call traveling wave a solution of \eqref{GP} of the form
$$
\Psi(x,t) =v(x_1-ct, x_\perp ), x_\perp=(x_2,...,x_d),
$$
where $c$ is the speed of propagation. The traveling waves play an important role in the long time dynamics of \eqref{GP}  (see e.g  \cite{chiron2,bethuelgravejatsmetsdarksoliton,gravejat2015asymptoticblack,chiron2021uniquenessvortex,chiron2017travelingequationnonzero}).  The profile $v$  solves the equation
\begin{equation}
\label{TWc}
\tag{TWc}
    -ic\partial_1v+\Delta v+v(1-|v|^2)=0.
\end{equation}
Using the Madelung transform, the associated  speed of sound for \eqref{GP}, around the constant solution $\psi=1$, is $\sqrt{2}$ (a rescaling changes the quantity $\sqrt{1g'(1)}$  into $1$).
The transonic limit was first studied by physicists (see \cite{jones1986motions,jones1982motions}). In the one-dimensional case, equation \eqref{TWc} is integrable with elementary computations. Solutions to \eqref{TWc} are related to the soliton of the Korteweg-de Vries equation
\begin{equation}
    \label{KdV}
    \tag{KdV}
    \partial_t \psi+\psi \partial_1\psi+ \partial_1^3 \psi=0.
\end{equation}
Indeed in  the transonic limit $c \rightarrow \sqrt{2}$, the traveling waves converge, up to rescaling, to the \eqref{KdV}
soliton (see \cite{BGS3,Chiron}). A similar result exists in the two-dimensional case: for any  $p>0$, there exists a non-constant finite energy solution $v_p$ to \eqref{TWc} with $ P(v_p)=\frac{1}{2} \int_{\mathbb{R}^d}\mathfrak{Re}\big( i \nabla \Psi( \overline{\Psi-1})\big)  = p $ (see \cite{BGS} theorem 1 and the survey \cite{BGS3} for properties of traveling waves). The convergence, in the transonic limit,  of those minimizing traveling waves for
the two-dimensional Gross-Pitaevskii equation towards a ground state of
the Kadomtsev-Petviashvili equation was obtained by Bethuel-Gravejat-Saut (see \cite{BGS2}). The Kadomtsev-Petviashvili equation
\begin{equation}
\tag{KPI}
\label{KP1}
    \partial_t \psi + \psi \partial_1 \psi +\partial_1^3\psi -\partial_1^{-1}\big(\partial_2^2\psi \big)=0,
\end{equation}
is a higher dimensional generalization of the Korteweg de Vries equation with energy 
\begin{equation}
    E_{KP}(\psi)=\frac{1}{2}\int_{\Ra^2}(\p_1 \psi)^2+\frac{1}{2}\int_{\Ra^2}(\p_1^{-1}(\p_2 \psi))^2-\frac{1}{6}\int_{\Ra^2}( \psi)^3.
\end{equation}
It is a well-known asymptotic model for the propagation of weakly transverse dispersive
waves \cite{kadomtsev1970stability}.
Solitary waves are localized solutions to \eqref{KP1} of the form $\psi(x,t)=\omega(x_1-\sigma t,x_2)$, where $\omega$ belongs to the energy space for \eqref{KP1}, i.e. the space $Y(\mathbb{R}^2)$ (see \cite{BouardSaut3}) defined
as the closure of $\partial_1\mathcal{C}^{\infty}_c(\mathbb{R}^2)$ for the norm
$$
\|\partial_1f\|_{Y(\mathbb{R}^2)}=( \| \nabla f \|^2_{L^2(\mathbb{R}^2)}+\| \partial_1^2  f \|^2_{L^2(\mathbb{R}^2)} )^{\frac{1}{2}}.
$$
The equation of a
solitary wave $\omega$ of speed $\sigma= 1$ is given by    
\begin{equation}
\label{SW}
\tag{SW}
  \partial_1 \omega -\omega \partial_1 \omega -\partial_1^3 \omega + \partial_1^{-1}\big(\partial_2^2 \omega \big)=0.
\end{equation} 
Given any $\sigma >0$, the scale change $\overset{\sim}{\omega}(x,y)=\sigma\omega \big( x\sqrt{\sigma},y\sigma\big)$ transforms any solution of equation \eqref{SW} into  a solitary wave $\overset{\sim}{\omega}$ of speed $\sigma$. The strategy in \cite{BGS2}  is to rewrite \eqref{GP} as an hydrodynamical system using Madelung transform, then rewrite the new equationss as a Kadomtsev-Petviashvili equation with some remainder. This transonic limit convergence result has been generalized in dimension two and three by Chiron and Mari{\c{s}} in \cite{chiron2014rarefaction} for a large class of nonlinearities.\newline 
In the same spirit, it is proved in \cite{chiron114} that solutions of \eqref{E-K} with well-prepared initial data converge,  in a long wave asymptotic regime,  to a solution of the Kadomtsev-Petviashvili (Korteweg de Vries in the one-dimensional case) equation (See also \cite{bethuel20092,bethuel2009korteweg,chiron2010kdvequationschro,chiron2014error}).
\par
The main focus of this paper is the convergence in the  transonic limit of the Euler Korteweg two dimensional traveling waves to a ground state of \eqref{KP1}. We also obtain with a more elementary argument, in dimension one, the convergence of the Euler-Korteweg soliton toward the soliton to the Korteweg-de-Vries equation.
\paragraph{Heuristic}Let us do the following formal computation. We let $\rho_p(x_1,x_2)=1+ \varepsilon_p^2\eta_p(z_1,z_2) $, $\phi_p=\varepsilon_p\theta_p(z_1,z_2)$, with $z_1=\varepsilon_p x_1$, $z_2=\varepsilon_p^2 x_2$
and $c_p=\sqrt{1-\varepsilon_p^2}$.
Then the first line of (\ref{E-K simple}) rewrites
$$
-c_p\partial_1\eta_p+\partial_1^2\theta_p+\varepsilon_p^2 \big(\partial_2^2\theta_p+\eta_p\partial_1^2 \theta_{p}+\partial_1 \eta_p \partial_1 \theta_p \big)= O(\varepsilon_p^4),
$$
furthermore, by Taylor expansion we have $ g=\varepsilon_p^2\eta_p +g''(1)\varepsilon_p^4 \frac{\eta_p^2}{2}+ O(\varepsilon_p^4)$, then the second lines of \eqref{E-K simple} rewrites
$$
-c_p \partial_1\theta_p+ \eta_p + \varepsilon_p^2 \bigg(\frac{g''(1) \eta_p^2+(\partial_1\theta_p)^2}{2}-K(1)\partial_1^2\eta_p\bigg)= O (\varepsilon_p^4).
$$
At first order we have $\partial_1\theta_p=\eta_p + O (\varepsilon_p^2)$, so that these functions should have the same limit. Then, multiplying the first equation by $c_p$ and applying the operator $\partial_1$ to the second equation, we obtain
$$
\partial_1 \eta_p + \partial_2^2\partial_1^{-1}\eta_p+ (3+g''(1))\eta_p \partial_1\eta_p-K(1) \partial_1^3\eta_p= O(\varepsilon_p^2).
$$
Finally, we let 
$$
\eta_p=-\frac{1}{3+g''(1)}N_p \bigg(\frac{x_1}{\sqrt{K(1)}},\frac{x_2}{\sqrt{K(1)}} \bigg).
$$
Then $N_p$ is an (approximate) solution to \eqref{SW}.
\paragraph{Main result}
Let $(\rho_p,\phi_p)$ be the  solution given by theorem 1, we consider
\begin{equation}
     \eta_p(x_1,x_2) =\rho_p(x_1,x_2)-1,
\end{equation}
and the rescaled functions
\begin{equation}
\left\{\begin{array}{lll}
\displaystyle
\theta_p(x_1,x_2)  &=&  \displaystyle
-\frac{1}{\gamma\varepsilon_p\sqrt{K(1)}}\phi_p \bigg( \frac{\sqrt{K(1)}x_1}{\varepsilon_p}, \frac{\sqrt{K(1)}x_2}{\varepsilon_p^2} \bigg), \\
\displaystyle
N_p(x_1,x_2)  &=&  \displaystyle
 -\frac{1}{\gamma\varepsilon_p^2}\eta_p \bigg( \frac{\sqrt{K(1)}x_1}{\varepsilon_p}, \frac{\sqrt{K(1)}x_2}{\varepsilon_p^2}\bigg).\end{array}
 \right.
\end{equation}
where 
\begin{equation}
\label{epsilon}
\varepsilon_p=\sqrt{1-c_p^2}\;\; \text{and} \;\;  \gamma=\frac{1}{g''(1)+3}.
\end{equation}
Our main theorem is 
\begin{thm}\label{thm de l'article} Under the assumption $\Gamma:=3+g''(1) \neq 0 $,
let $(p_n)_{n \in\mathbb{N}}$ such that, $p_n \rightarrow 0$. Then, there exists a  ground state $N_0$ of \eqref{KP1}, such that,  up to a subsequence, we have 
\begin{equation}
    N_{p_n} \rightarrow N_0 \; \text{in} \; W^{k,q}(\mathbb{R}^2),  \text{when}   \;n \rightarrow \infty.
\end{equation}
and 
\begin{equation}
    \partial_1 \theta_{p_n} \rightarrow N_0 \; \text{in} \; W^{k,q}(\mathbb{R}^2), \text{when}  \;n \rightarrow \infty.
\end{equation}
for any $k \in \mathbb{N}$ and any $1<q\leq \infty$.
\end{thm}
\begin{rem}
The proof in the manuscript follows the same lines of \cite{BGS2}, but is more involved at a technical level
because of the extension to arbitrary non-linearities $g$ and $K$. It provides an
extension of recent results on the nonlinear Schrödinger equations with non-zero condition at
infinity towards the wider, but still physically relevant, class of the Euler-Korteweg systems. See \cite{chiron2019smoothpacherie} for the existence of smooth branch of travelling waves for the Euler-Korteweg  equation converging to the first lump in the transonic limit. It is not known that the solitons of Theorem \ref{existencethm} are the same as those in \cite{chiron2019smoothpacherie},
thus the result of theorem \ref{thm de l'article} is not contained in \cite{chiron2019smoothpacherie} (and conversely) .
To obtain the convergence of the full sequence, it is sufficient that the limit $N_0$ is unique  but it is a difficult and open problem (see \cite{liu2019nondegeneracy}). 
\end{rem}
\begin{rem}
 A condition similar to $\Gamma \ne 0$  is highlighted in \cite{chiron2017travelingequationnonzero} for the Schrödinger equation.
As has been proven for
the Schrödinger equation (see \cite{chiron2014rarefaction}), it is possible that a similar result exist in dimension three. Let us recall that the existence of solitons in dimension 3 is an open problem for Euler-Korteweg.
\end{rem}
We will also prove a similar result in the one-dimensional case. That is,    the  \eqref{E-K} solitons  converge,  up  to rescaling,  to  the \eqref{KdV} soliton, in the transonic limit (see the appendix \ref{appendix diml 1} for a precise statement). Moreover as the computation are simpler in dimension 1 we are able to compute the transonic limit for $\Gamma=0$ and a new nondegeneracy condition (see proposition \ref{proposition gamma = 0 }). In this case, the limit is not a solution of \eqref{KdV} but of (gKdV).
\begin{proposition}
\label{dim 1 convergence gamma pas 0}
Under the conditions $$
g(1) = 0, \;g'(1) = 1,\; \Gamma \ne 0 . $$
there exists  $(\rho,u)$ global solution of
\begin{equation}
\tag{E-K}
\label{ekk}
  \left\{
      \begin{aligned}
        &-c\rho'+(\rho u)'=0, \qquad \qquad \quad \quad \qquad \quad \quad  (1) \\
        &\-cu'+ \bigg( \frac{u^2}{2} \bigg)'+g'= \bigg( K(\rho) \rho''+\frac{1}{2}K'(\rho)\rho'^2 \bigg)' \quad (2), \quad x\in \mathbb{R},
      \end{aligned}
    \right.
\end{equation}
with $\|\rho - 1\|_{L^{\infty}}\underset{\varepsilon \rightarrow 0}{\rightarrow}0$.
Moreover if we let 
\begin{equation}
    \rho-1=-\varepsilon^2 \gamma r_{\varepsilon}\bigg( \frac{\varepsilon x}{\sqrt{K(1)}} \bigg),
\end{equation}
then for any $k \in \mathbb{N}$
$$
\|r_{\varepsilon}^{(k)} -N^{(k)}\|_{L^{\infty}(\mathbb{R})} \underset{\varepsilon \rightarrow 0}{\longrightarrow} 0,
$$
where 
\begin{equation}
    N(x)=\frac{3}{ch^2(\frac{x}{2})},
\end{equation}
is the classical soliton to the Korteweg-de-Vries equation
\begin{equation}\label{Kdv} \tag{KdV}
    - \psi'+\psi \psi'+\psi'''=0.
\end{equation}
\end{proposition}
\begin{rem}
The argument proposed in \cite{Chiron} for the Schrödinger equation should
extend to our framework, nevertheless we propose an alternative proof in section \ref{appendix diml 1}.
\end{rem}
It is also possible to describe the case $\Gamma=0$:
\begin{proposition}\label{proposition gamma = 0 }
Under the conditions $$
g(1) = 0, \,g'(1) = 1,\, \Gamma = 0, \, \Gamma':=g'''(1)-12<0 , $$
There exists $(\rho^{\pm},u^{\pm})$ global solution of \eqref{ekk} with $\|\rho^{\pm} - 1\|_{L^{\infty}}\underset{\varepsilon \rightarrow 0}{\rightarrow}0$ .
Moreover if we let 
$$
  \rho^{\pm}-1=\varepsilon \gamma' r_{\varepsilon}^{\pm}\bigg( \frac{\varepsilon x}{\sqrt{K(1)}} \bigg),
$$
with $\gamma'=\frac{1}{\sqrt{12-g'''(1)}}$ then for any $k \in \mathbb{N}$ 
$$\|r^{\pm(k)}_{\varepsilon} -(w^{\pm(k)})\|_{L^{\infty}(\mathbb{R})} \underset{\varepsilon \rightarrow 0}{\longrightarrow} 0.$$ where $w^{\pm}(x)=\pm\frac{\sqrt{12}}{ch(x)}$ are the two opposite soliton of the focusing modified Korteweg de Vries equation
\begin{equation}
    \tag{mKdV}
    \psi'-\frac{1}{2}\psi^2\psi'=\psi'''.
\end{equation}
\end{proposition}
\paragraph{Organization of the article}
In section 2 we introduce the notations and recall the properties of solitary waves solutions to \eqref{KP1}. In section 3 we prove that $N_p$ and $\partial_1 \theta_p$ converge and have the same limit. In section 4, we prove that Sobolev bounds for $N_p$  give bound for  $\theta_p$. In section 5,  using the  Taylor expansion in \eqref{E-K simple} with respect to  $\varepsilon$  we obtain the \eqref{SW} equation with some remainder. 
Using Fourier transform we obtain Sobolev bounds for  $N_p$ and $\partial_1\theta_p$, in section 6 and 7. Finally, we end the proof of theorem 2 in section 8.
\section{Notations, functional spaces and properties of solution to \eqref{KP1} }
\paragraph{Functional spaces}
Let  $p \in [1,+\infty]$, $k \in \mathbb{N}$, $0<\alpha<1$ and $\Omega$ be an 
 smooth open subset of $\mathbb{R}^d$. 
We denote by $W^{m,p}(\Omega)$, $H^{m}(\Omega)=W^{m,2}(\Omega)$ the usual Sobolev spaces.
For $s \geq 0$,  we define
$$
H^s(\mathbb{R}^d)= \left\{ u \in L^2(\mathbb{R}^d) \bigg| \|u\|^2_{H^s(\mathbb{R}^d)}= \int_{\mathbb{R}^d} (1+|\xi|^2)^s|\hat{u}(\xi)|^2d\xi < \infty \right\}.
$$
We define
$$
\dot{H}^s(\mathbb{R}^d)= \left\{u\bigg|\ \hat{u} \in L_{loc}^1(\mathbb{R}^d),\  \int_{\mathbb{R}^d} |\xi|^{2s}|\hat{u}(\xi)|^2d\xi < \infty \right\}.
$$
$C^{0,\alpha}(\Omega)$ is the space of bounded $\alpha$-Hölder continuous  function on $\Omega$.
We define
$$C^k_0(\mathbb{R}^d) = \{ u \in C^k(\mathbb{R}^d): \lim_{|x| \rightarrow \infty} | \partial^{\alpha} u(x)|=0, \; \forall \alpha \in \mathbb{N}^d,\, |\alpha| \leq k \},$$ with the norm
$$
||u||_{C^k_0(\mathbb{R}^d)}= \sum_{|\alpha| \leq k } \sup_{\mathbb{R}^d}|\partial^{\alpha}u|.
$$
\paragraph{Sobolev embedding} We recall the Sobolev embeddings,
 \begin{equation}
\label{embedding sobolev}
  \forall\, kq_1<d ,\;W^{k,q_1}(\mathbb{R}^d) \hookrightarrow L^{q_2}(\mathbb{R}^d), \, q_2=\frac{dq_1}{d-kq_1},\; 
 \end{equation}
 \begin{equation}
 \label{holder espace injection}
     \forall \, d<q<\infty, \, W^{1,q}(\Omega) \hookrightarrow C^{0,1-\frac{d}{q}}(\Omega).
 \end{equation} 
\begin{equation}
\label{injection ck sobolev}
\forall (k-k')q>d,\; W^{k,q}(\Ra^d) \hookrightarrow C^{k'}_0(\Ra^d).
\end{equation} 
Moreover if $\Omega$ is bounded the embedding \begin{equation}
\label{injection compacte holder}
W^{1,q}(\Omega) \hookrightarrow C^{0,1-\frac{d}{\alpha}}(\Omega), d<\alpha<q,     
\end{equation}is compact.
In particular
\begin{equation}
\label{lemme9}
    H^{\frac{d}{2}}(\mathbb{R}^d)  \hookrightarrow \, L^q(\mathbb{R}^d),\; 2 \leq q < \infty,
\end{equation}
\begin{equation}
\label{rappelsuite}
     \forall \; 0\leq s <\frac{d}{2}, \,
H^{s}(\mathbb{R}^d) \hookrightarrow L^q(\mathbb{R}^d), \;  2 \leq q \leq \frac{2d}{d-2s},
\end{equation}
\begin{equation}
\label{theorème de morrey}
    \forall s > \frac{d}{2} +k, \,
H^s(\mathbb{R}^d) \hookrightarrow C^k_0(\mathbb{R}^d).
\end{equation}
\paragraph{Basic convolution results}
let $(f,g) \in H^s(\mathbb{R}^d) \times L^1(\mathbb{R}^d)$, we have
\begin{equation}
\label{jenaimarrre}
  ||f \ast g ||_{H^s(\Ra^d)} \leq ||f||_{H^s(\Ra^d)}||g||_{L^{1}(\Ra^d)}.  \end{equation}
We recall a result on Fourier multipliers due to Lizorkin.
\begin{thm} \label{lizorkin}(\cite{lizorkin}).  \label{theorem5.1}
Let $\widehat{K}$ be a bounded function in $C^2(\mathbb{R}^2 \backslash \{ {0} \})$ and assume that
$$
\xi_1^{k_1}\xi^{k_2}_2\partial_1^{k_1}\partial_2^{k_2}\widehat{K}(\xi) \in L^{\infty}(\mathbb{R}^2),
$$
for any integer $0 \leq k_1,k_2 \leq 1$ such that $k_1+k_2\leq 2 $. Then, $\widehat{K}$ is a multiplier from $L^q(\mathbb{R}^2) \;\text{to}\; L^q(\mathbb{R}^2)$ for any $1<q<\infty$, i.e.  there exists a constant $C(q)$, depending only on q, such that 
$$
|| K * f ||_{L^q(\mathbb{R}^2)} \leq C(q) M( \widehat{K}) ||f||_{L^q(\mathbb{R}^2)}, \quad \forall f \in L^q(\mathbb{R}^2).
$$
where we denote
$$
M( \widehat{K})= sup \{ |\xi_1|^{k_1}|\xi_2|^{k_2}|\partial_1^{k_1}\partial_2^{k_2}\widehat{K}(\xi)|, \xi \in \mathbb{R}^2, 0\leq k_1\leq1, 0\leq k_2\leq1, 0\leq k_1+k_2\leq2  \}.
$$
\end{thm}
\paragraph{Existence and properties of solitary wave solutions to \eqref{KP1}
}We recall some results on the Kadomtsev-Petviashvili Equation. A ground state is a solitary wave that minimizes the action 
$$
S(\omega)=E_{KP}(\omega)+\frac{\sigma}{2}\int_{\Ra^2}\omega^2,
$$
among all non-constant solitary waves of speed $\sigma$. The constant $S_{KP}$ denotes the action $S(\omega)$ of the ground states $\omega$ of speed $\sigma=1$. We will denote by $\mathcal{G}_\sigma$ the set of the ground state of speed $\sigma$. The ground states solutions are characterized as minimizers of energy constrained by constant $L^2$-norm. Let $\mu > 0$, then the minimization problem
\begin{equation}
\label{probleme minimization}
\tag{$P_{KP}(\mu)$}
\mathcal{E}^{KP}(\mu)=\inf \left\{ E_{KP}(\omega), \omega \in Y(\mathbb{R}^2) , \int_{\mathbb{R}^2}|\omega|^2=\mu \right\},
\end{equation}
has at least one solution. Moreover there exists $\sigma$ such that the set of minima $\mathcal{E}^{KP}(\mu)$ is exactly equal to $\mathcal{G}_\sigma$ (see De Bouard and Saut \cite{bouardsaut2}). For $\sigma=1$ we have $\mu=\mu^*=3S_{KP}$. Since it was proved by making use of the concentration-compactness principle of P.L. Lions (see \cite{LionsCC}), we have the compactness of minimizing sequences.
\begin{thm}(\cite{bouardsaut2})
Let $\mu >0$, and let $(\omega_n)_{n \in\mathbb{N}}$ be a minimizing 
sequence to
\eqref{probleme minimization} 
in 
$Y(\Ra^2)$.
Then, there exists some points $(a_n)_{n \in \mathbb{N}}$ 
and a function $N \in Y(\mathbb{R}^2)$ such that up to some subsequence, 
$$
\omega_n(.-a_n) \rightarrow N\; {\rm in} \; Y(\Ra^2), \; {\rm as} \ n\rightarrow \infty.
$$
N is solution to the minimization problem \eqref{probleme minimization} and thus is a ground state for \eqref{KP1}.
\end{thm}
Using a scaling argument it is possible to compute $\mathcal{E}^{KP}.$
\begin{lemme}(\cite{BGS2})\label{lemme 2.1} Let $N \in Y(\Ra^2)$. Given any $\sigma > 0$,
the function $N_{\sigma}(x_1,x_2)=\sigma N( \sqrt{\sigma} x_1,\sigma x_2)$
is a minimizer for $\mathcal{E}^{KP}(\sqrt{\sigma}\mu^*)$ if and only if N is a minimizer for $\mathcal{E}^{KP}(\mu^*)$. In particular we have $$\mathcal{E}^{KP}(\mu)=-\frac{\mu^3}{54S_{KP}^2}, \; \forall \mu > 0.
$$
Moreover, $N_{\sigma}$ and N are ground states for \eqref{KP1} , with respective speed $\sigma$, and 1. We have the relation $\sigma=\frac{\mu^2}{(\mu^*)^2}.$
\end{lemme}
Finally, in our proof, we will have sequences $(\omega_n)_{n \in \mathbb{N}}$ which are not minimizing sequences for \eqref{probleme minimization} but satisfy
\begin{equation}
\label{presque sequences minimisante}
    E_{KP}(\omega_n)\rightarrow \mathcal{E}^{KP}(\mu),\; {\rm and} \; \int_{\Ra^2} \omega_n^2 \rightarrow \mu, \ {\rm as} \; n \rightarrow \infty,
\end{equation}
for some positive number $\mu$.
\begin{proposition}(\cite{BGS2})\label{compacite solution KP} Let $\mu>0$, and $(\omega_n)_{n \in \mathbb{N}}$ denote a sequence of functions in $Y(\mathbb{R}^2)$ satisfying \eqref{presque sequences minimisante} for $\mu$. Then, there exists some sequence $(a_n)_{n \in \mathbb{N}}$ and a ground state solution $N$, with speed $\sigma=\frac{\mu^2}{(\mu^*)^2}$ such that, up to some subsequence,
\begin{equation}
    \omega_n(.-a_n)\rightarrow N \; {\rm in} \; Y(\Ra^2), \; {\rm as} \; n \rightarrow \infty.
\end{equation}
\end{proposition}
\paragraph{Reformulation}In \cite{debouardsaut}, the authors use a new formulation for the solitary wave equation.
Applying operator $\p_1$ to \eqref{SW} we have 
$$
 \partial_1^4\omega- \Delta \omega +\frac{1}{2} \partial_1^2(\omega^2)=0.
$$
Let
 \begin{equation}
     \label{2.4}
     \widehat{K}_0(\xi)=\frac{\xi_1^2}{|  \xi|^2+\xi_1^4},\; \xi \in \mathbb{R}^2.
 \end{equation}
 and $\omega \in Y(\mathbb{R}^2) $. Then,
 $\omega$ is a solution to \eqref{SW} if and only if
 \begin{equation}
 \label{SWW}
     \omega=\frac{1}{2}K_0 \ast \omega^2.
 \end{equation}
\section{Weak convergence in $L^2$}
The aim of this section is to prove that $N_p$ and $\partial_1\theta_p$ have the same limit, and compute the convergence speed of $N_p-\partial_1 \theta_p$ towards 0.
\subsection{Rescaling and energy}
Let $(\rho_p,\phi_p)$ the solutions given by theorem \ref{existencethm}. 
As in \cite{BGS2}, we consider  rescaled functions and use anisotropic space variables. let
\begin{equation}
\varepsilon_p=\sqrt{1-c_p^2}\;\; \text{and} \;\;  \gamma=\frac{1}{g''(1)+3},
\end{equation}
\begin{equation}
\label{chgmt variable1}
    \theta_p(x_1,x_2)=-\frac{1}{\gamma\varepsilon_p\sqrt{K(1)}}\phi_p \bigg( \frac{\sqrt{K(1)}x_1}{\varepsilon_p}, \frac{\sqrt{K(1)}x_2}{\varepsilon_p^2} \bigg),
\end{equation}
and
\begin{equation}
\label{chgmt variable2}
    N_p(x_1,x_2)=-\frac{1}{\gamma\varepsilon_p^2}\eta_p \bigg( \frac{\sqrt{K(1)}x_1}{\varepsilon_p}, \frac{\sqrt{K(1)}x_2}{\varepsilon_p^2} \bigg).
\end{equation}
We can assume, up to a translation,  using \eqref{proposition 3.3}
\begin{equation}
\label{minimal}
     \exists C>0, \; N_p(0) >C, \; \forall p_0>p>0.
\end{equation}
\begin{proposition}
\label{limitecommune}
Let $(p_n)_{n \in\mathbb{N}}$ a sequence such that $p_n \rightarrow 0$. Then,  there exists $N_0 \in L^2(\mathbb{R}^2)$ such that, up to a subsequence,
\begin{equation}
\label{limiteN}
    N_{p_n} \rightharpoonup N_0 \; {\rm in} \; L^2(\mathbb{R}^2)
\end{equation}
\begin{equation}
\label{limitetheta}
    \partial_1\theta_{p_n} \rightharpoonup N_0 \; {\rm in} \; L^2(\mathbb{R}^2)
\end{equation}
Moreover there exists some positive constant $C$, not depending on p, such that
\begin{equation}
\label{Ntheta}
     \int_{\mathbb{R}^2}(N_p-\partial_1 \theta_p)^2dx \leq C \varepsilon_p^{\frac{1}{2}}.
\end{equation}
\end{proposition}
\begin{rem}
In section \ref{patate}, we will prove that $N_0$ is a ground state of \eqref{KP1}.
\end{rem}
\begin{proof}[Proof of \eqref{limiteN} an \eqref{limitetheta}]
Since $G''(1)=g'(1)=1$, $G'(1)=g(1)=G(1)=0$, there exists $\delta>0$ such that $G(\rho) \geq \frac{(\rho-1)^2}{3}$  for any $\rho \in ]1+\delta, 1-\delta[$. Then for $p$ small enough, using \eqref{2.3} and the definition \eqref{energie E K}, we have
\begin{equation}
\label{1.10}
    \int_{\mathbb{R}^2} (\eta_p \big(
x_1, x_2 \big))^2dx \leq M E(\rho_p,\phi_p).
\end{equation}
Thus, we deduce from \eqref{energie-moment} and \eqref{epsilon} that
\begin{equation}
\label{controle de N  en normel L2}
\begin{split}
\int_{\mathbb{R}^2}(N_p)^2 & = \frac{1}{\gamma^2 \varepsilon_p^4}\int_{\mathbb{R}^2} \bigg(\eta_p \bigg(
\frac{\sqrt{K(1)}x_1}{\varepsilon_p}, \frac{\sqrt{K(1)}x_2}{\varepsilon_p^2} \bigg)\bigg)^2 \\
&\leq \frac{1}{K(1)\gamma^2\varepsilon_p}\int_{\mathbb{R}^2} (\eta_p \big(
x_1, x_2 \big))^2
\\
& \leq \frac{1}{\varepsilon_p} C E(\rho_p,\phi_p) \qquad  \\
& \leq C.
\end{split}
\end{equation}
Then, using Banach-Alaoglu theorem, there exists a function $N_0 \in L^2(\mathbb{R}^2)$ such that, up to some subsequence  \begin{equation}
    N_{p_n} \rightharpoonup  N_0 \; \text{in} \; L^2(\mathbb{R}^2).
\end{equation}
The convergence of $\partial_1\theta_{p_n}$ is a consequence of \eqref{Ntheta}.
In order to complete the proof of proposition \ref{limitecommune}, it only remains to prove \eqref{Ntheta}. This requires to use rescaled energy and Pohozaev estimates, so that \eqref{Ntheta} is postponed to section 3.3.
\end{proof}
\begin{lemme}\label{energie nouvelle fonction}
The energy can be expressed
in terms of the new functions as 
\begin{align*}
E(\rho_p,\phi_p)=\frac{K(1)\gamma^2\varepsilon_p}{2} \bigg(E_0(N_p, \theta_p) +\varepsilon_p^2 (E_2(N_p,\theta_p)))+
  \varepsilon_p^4 (E_4(N_p,\theta_p)) \bigg),
\end{align*}
with
$$E_0(N_p, \theta_p)=\int_{\mathbb{R}^2}  N^2_p+(\partial_1 \theta_p)^2, $$

\begin{align*}
E_2(N_p, \theta_p)=2 \bigg(\int_{\mathbb{R}^2}    \frac{(\partial_1 N_p(x))^2}{2}+\frac{(\partial_2 \theta_p (x))^2}{2}- \frac{\gamma}{6}\big( 3N_p(\partial_1 \theta_p)^2+ g''(1)N_p^3)dx\bigg),
\end{align*}

\begin{align*}
E_4(N_p,\theta_p)=\int_{\mathbb{R}^2}& K\bigg(1-\gamma\varepsilon_p^2N_p (x) \bigg) \frac{1}{\sqrt{K(1)}^2}(\partial_2N_p(x))^2-
\gamma N_p(x)(\partial_2 \theta_p (x))^2\\&+N_p(x) j_p(x)\bigg( 
 (\partial_1N_p(x))^2 + ( \varepsilon_p)^2 (\partial_2 N_p(x))^2 \bigg)\\&+\gamma^4  N_p^4(x)l_p(x) dx,
\end{align*}
with $j_p$ and $l_p$ some functions smooth and bounded in $L^{\infty}$ uniformly in $p$. 
For the momentum we have
$$
p=\varepsilon_p\gamma^2K(1)\int_{\Ra^2}N_p(x)\partial_1\theta_p(x)dx.
$$
\end{lemme}
\begin{proof}Since $N_p\simeq \p_1 \theta_p$, then passing to the limit in $p$ we have $E_0(N_p,\theta_p)\simeq2 \int_{\Ra^2}N_0^2$ and $E_2(N_p,\theta_p)\simeq 2 E_{KP}(N_0)$. Later, to prove that the weak limit $N_0$  is a
solution to \eqref{SW},  we will use \eqref{probleme minimization}.
This is a direct but tedious computation. The functions $j_p$ and $l_p$ are given by
\begin{align*}
G(1+x)&=G(1)+ G'(1)x+G''(1)\frac{x^2}{2}+x^3l(x) \\
&=\frac{x^2}{2}+ x^3 l(x).
\end{align*}
So
$$
G(\rho_p(x))=\gamma^2\frac{1}{2} \varepsilon_p^4N^2_p(x')-\gamma^3\varepsilon_p^6 \frac{g''(1)N_p^3(x')}{6}+\gamma^4 \varepsilon_p^8 N_p^4(x')l(-\varepsilon^2\gamma N_p(x')),
$$
where $l$ is the third order remainder of Taylor expansion and $$x'=\bigg( \frac{\varepsilon_p x_1}{\sqrt{K(1)}}, \frac{\varepsilon_p^2x_2}{\sqrt{K(1)}} \bigg).$$
In view of \eqref{2.3},  the function $l(-\varepsilon_p^2\gamma N_p(x'))=l(\rho_{p}(x)-1)$ is bounded independently of p. To simplify we write
$$
G(\rho_p(x))=\gamma^2\frac{1}{2} \varepsilon_p^4N^2_p(x')-\gamma^3\varepsilon_p^6 \frac{g''(1)N_p^3(x')}{6}+\gamma^4 \varepsilon_p^8 N_p^4(x')l_p(x').
$$
Similarly, we write
$$
K\bigg(1-\gamma\varepsilon_p^2N_p (x') \bigg)=K(1)- \varepsilon_p^2N_p(x') j_p(x').
$$
\end{proof}
\subsection{Pohozaev’s identities}
\label{pouuuuuuuuuuuuuuuuuuuuuuuuuuuf}
We estimate now derivative terms in the energy using Pohozaev's identities.  
\begin{lemme}\label{pohozaev lemma}
\begin{equation}
    \left|\int_{\Ra^2}(\partial_2\phi_p)^2+  (\partial_2\rho_p)^2+  (\partial_1\rho_p)^2dx\right| \leq C\varepsilon_p^3,
\end{equation}
and
\begin{equation}
      \left|\int_{\Ra^2} \eta_p (\partial_1\phi_p)^2dx \right|\leq C \varepsilon_p^{\frac{3}{2}}, \; \left| \int_{\Ra^2}\eta_p (\partial_2 \phi_p)^2dx \right| \leq C \varepsilon_p^{\frac{7}{2}}.
\end{equation}

\end{lemme}
\begin{proof}
Thanks to \eqref{energie-moment}, we have 
$$
|E(\rho_p,\phi_p)-P(\rho_p,\phi_p)|<CP(\rho_p,\phi_p)^3.
$$
We recall Pohozaev's identities obtained in  \cite{audiard2017small} proposition 5.1
\begin{equation}
   E(\rho_p,\phi_p)= \int_{\Ra^2}\rho_p(\partial_2\phi_p)^2+ K(\rho_p) (\partial_2\rho_p)^2dx + c_pP(\rho_p,\phi_p),
\end{equation}
\begin{equation}
\label{pohozaev2}
    E(\rho_p,\phi_p)= \int_{\Ra^2} \rho_p(\partial_1\phi_p)^2+K(\rho_p)(\partial_1\rho_p)^2dx.
\end{equation}
Moreover we have 
\begin{equation}
\label{5.4}
    c_pP(\rho_p,\phi_p)=\int_{\Ra^2} \rho_p (\nabla \phi_p)^2dx.
\end{equation}
It follows from   \eqref{energie-moment} and \eqref{epsilon} that
\begin{align*}
   \left|\int_{\Ra^2}\rho_p(\partial_2\phi_p)^2+ K (\partial_2\rho_p)^2dx\right| &=|E(\rho_p,\phi_p)-cP(\rho_p,\phi_p)| \\
   &=|E(\rho_p,\phi_p)-P(\rho_p,\phi_p)+(1-\sqrt{1-\varepsilon_p^2})P(\rho_p,\phi_p)|
   \\&\leq |E(\rho_p,\phi_p)-P(\rho_p,\phi_p)|+\varepsilon_p^2\left|\frac{1-\sqrt{1-\varepsilon_p^2}}{\varepsilon_p^2} P(\rho_p,\phi_p)\right|
   \\&\leq CP(\rho_p,\phi_p)^3   ,
\end{align*}
and therefore using
$||\rho_p - 1||_{\infty} \rightarrow 0$ when $p \rightarrow 0$ we get
\begin{equation}
\label{f}
\left|\int_{\Ra^2}(\partial_2\phi_p)^2+  (\partial_2\rho_p)^2dx\right| \leq C P(\rho_p,\phi_p)^3.
\end{equation}
In view of (\ref{pohozaev2}) and (\ref{5.4}) we obtain

$$   \left| \int_{\Ra^2} K(\rho_p) (\partial_1\rho_p)^2dx \right| =\left|\int \rho_p (\partial_2\phi_p)^2dx+E(\rho_p,\phi_p)-c P(\rho_p, \phi_p)\right|,
$$
and then 
\begin{equation}
\label{e}
    \left| \int_{\Ra^2} (\partial_1\rho_p)^2dx \right| \leq C P(\rho_p,\phi_p)^3.
\end{equation}
Finally combining \eqref{e} and \eqref{f} we have
\begin{equation}
\label{cool}
\left|\int_{\Ra^2}(\partial_2\phi_p)^2+  (\partial_2\rho_p)^2+  (\partial_1\rho_p)^2dx\right| \leq C P(\rho_p,\phi_p)^3.
\end{equation}
Moreover by \eqref{pohozaev2}
\begin{equation}
\label{5.1}
    E(\rho_p,\phi_p)=\int_{\Ra^2} \rho_p(\partial_1\phi_p)^2+K(\rho_p)(\partial_1\rho_p)^2dx,
\end{equation}
thus
$$
\int_{\Ra^2} (\partial_1\phi_p)^2dx \leq C E(\rho_p, \phi_p).
$$
Combining \eqref{5.1} with \eqref{2.3}, we obtain
\begin{equation}
\label{etad1phi}
    \left|\int_{\Ra^2} \eta_p (\partial_1\phi_p)^2dx \right|\leq C \varepsilon_p^{\frac{3}{2}}.
\end{equation}
Similarly
\begin{equation}
\label{etad2phi}
     \left| \int_{\Ra^2}\eta_p (\partial_2 \phi_p)^2dx \right| \leq C \varepsilon_p^{\frac{7}{2}},
\end{equation}
this ends the proof of lemma \ref{pohozaev lemma}. 
\end{proof}
\subsection{Energy estimates}\label{pifpafpouf}
We are now in position to conclude the proof of proposition \ref{limitecommune}.
\begin{proof}[End of proof of proposition \ref{limitecommune}]First we have
\begin{align*}
&\int_{\Ra^2} \big(\partial_1N_p\big)^2= \frac{1}{\gamma^2\varepsilon_p^3}\int_{\Ra^2}(\partial_1 \rho_p)^2,&
 \int_{\Ra^2} \big(\partial_2\theta_p\big)^2= \frac{1}{K(1)\gamma^2\varepsilon_p^3}\int_{\Ra^2}(\partial_2 \phi_p)^2,\\&\int_{\Ra^2} \big(\partial_2N_p\big)^2= \frac{1}{\gamma^2\varepsilon_p^5}\int_{\Ra^2}(\partial_2 \rho_p)^2,&
\int_{\Ra^2} N_p(\partial_1 \theta_p)^2=- \frac{1}{K(1)\gamma^3\varepsilon_p^3}\int \eta_p (\partial_1 \phi_p)^2,\\&\int_{\Ra^2} N_p^3=-\frac{1}{\gamma^3K(1)\varepsilon_p^3}   \int_{\Ra^2} \eta_p^3,& \int_{\Ra^2}N_p(\partial_2 \theta_p)^2=- \frac{1}{\gamma^3\varepsilon_p^5K(1)} \int_{\Ra^2} \eta_p (\partial_2 \phi_p)^2 .
\end{align*}
Using  \eqref{cool} and \eqref{epsilon},  we have 
\begin{equation}
    \label{coool}
    \int_{\Ra^2} \big(\partial_2\theta_p\big)^2+\big(\partial_1N_p\big)^2 \leq C.
\end{equation}
Then with \eqref{2.3}, \eqref{1.10} and \eqref{etad1phi} we obtain
\begin{equation}
\label{E_1}
    E_2(N_p,\theta_p) \leq \frac{C}{\varepsilon_p^{\frac{3}{2}}}.
\end{equation}
Similarly, using \eqref{etad2phi}, \eqref{cool}, \eqref{2.3} and by definition of $E_4$, we have
 \begin{equation}
     E_4(N_p,\theta_p) \leq  \frac{C}{\varepsilon_p^3}.
 \end{equation}
We deduce that 
\begin{equation}
\begin{split}
\bigg|E(\rho_p,\phi_p)- \frac{\gamma^2\varepsilon_pK(1)}{2}E_0(\rho_p,\phi_p) \bigg|&=
\bigg|E(\rho_p,\phi_p)- \frac{\gamma^2\varepsilon_pK(1)}{2}\int_{\Ra^2}(N_p)^2+(\partial_1\theta_p)^2 \bigg|\\ &\leq C \varepsilon_p^\frac{3}{2}.
\end{split}
\end{equation}
Thus, by \eqref{energie-moment} we have
\begin{align*}
    \int_{\Ra^2}(N_p-\partial_1 \theta_p)^2&= E_0(N_p,\theta_p)- 2 \int_{\Ra^2} N_p\partial_1\theta_p \\& \leq \frac{2E(\rho_p,\phi_p)}{\gamma^2
    K(1)\varepsilon_p} -\frac{2}{\varepsilon_p\gamma^2
    K(1)} p +C \varepsilon_p^\frac{1}{2}\\
    &\leq C_1 \varepsilon_p^{\frac{1}{2}}.
\end{align*}
This completes the proof of estimate \eqref{Ntheta} and Proposition \ref{limitecommune}. 
\end{proof}
\section{Elliptic estimates on $\theta_p$ }
In the present section, we aim to prove that $N_p$ bounds $\theta_p$ in Sobolev spaces. More precisely, we prove the following proposition:
\begin{proposition}
\label{lemme 4.6}
Let $1 < q < \infty$ there exists a constant $ C(q)$ depending on q, but not on $p$, such that
\begin{equation}
\label{controle theta}
    ||\partial_1 \theta_p ||_{L^q(\mathbb{R}^2)} +
\varepsilon_p|| \partial_2 \theta_p ||_{L^q(\mathbb{R}^2)} \leq C(q) ||N_p||_{L^q(\mathbb{R}^2)},
\end{equation}
for p small enough.
Moreover for any $\alpha \in \mathbb{N}^2$ let
$$
\Gamma_p(q,\alpha)= ||\partial^{\alpha} \partial_1 \theta_p ||_{L^q(\mathbb{R}^2)}+\varepsilon_p || \partial^{\alpha} \partial_2 \theta_p||_{L^q(\mathbb{R}^2)},
$$
then there exists $C(q,\alpha)$ such that
\begin{equation}
\label{4.26}
    \Gamma_p(q,\alpha) \leq C(q, \alpha)  \bigg(|| \partial^{\alpha}N_p||_{L^q(\mathbb{R}^2)} + \varepsilon_p^2 \sum_{0 \leq \beta < \alpha}|| \partial^{\beta} N_p||_{L^{\infty}(\mathbb{R}^2)} \Gamma_p(q, \alpha- \beta) \bigg).
\end{equation}
\end{proposition}
First of all in view of theorem \ref{existencethm} and lemma 4.2 in \cite{audiard2017small} we have:
\begin{lemme}\label{régularité}
For any $k \in \mathbb{N}$, $q \in [2, \infty[$ we have
\begin{equation}
\label{regularite}
    (\eta_p,\nabla \phi_p) \in W^{k,q}(\mathbb{R}^2) \cap C^{\infty} (\mathbb{R}^2).
\end{equation}
Moreover for any $\alpha \in \mathbb{N}^2$, there exists $C(\alpha)>0$ not depending on $p$ such that
\begin{equation}
\label{controleee}
    ||(\partial^{\alpha}\eta_p,\partial^{\alpha}\nabla \phi_p)||_{L^{\infty}(\mathbb{R}^2)} \leq C(\alpha).
\end{equation}
\end{lemme}
Later, using Lizorkin theorem, we will have $N_p$, $\partial_1 \theta_p$ and $\p_2 \theta_p$ in $W^{k,q}(\Ra^2)$, for any $1<q<2$ and $k \in \mathbb{N}$. Thus the quantity $\Gamma_p(q,\alpha)$ is finite for any $1<q<\infty,\alpha \in \mathbb{N}^2$. This is the reason 
why in proposition \ref{lemme 4.6}, we let $1<q<\infty$.
\begin{proof}[Proof of proposition \ref{lemme 4.6}]
We have
\begin{equation}
    \label{1.12}
    -c_p \partial_1 \rho_p =-\text{div}(\rho_p \nabla \phi_p),
\end{equation}
then
$$
\Delta \phi_p = c_p
\partial_1 \eta_p - \text{div}(\eta_p \nabla \phi_p),
$$
and for any $\alpha \in \mathbb{N}^2$  $$
\Delta (\partial^{\alpha}\phi_p) = c_p \partial_1 \partial^{\alpha}\eta_p - \text{div}(\partial^{\alpha}(\eta_p \nabla \phi_p)).
$$
Thus, using elliptic estimates (see \cite{gilbarg1977elliptic}), there exists some constant $C(q)$ such that
\begin{equation}
\label{elliptique lol}
    || \nabla (\partial^{\alpha} \phi_p) ||_{L^q(\mathbb{R}^2)} \leq C(q) \big( 
    ||\partial^{\alpha} \eta_p ||_{L^q(\mathbb{R}^2)}+ ||\partial^{\alpha}(\eta_p \nabla \phi_p)||_{L^q(\mathbb{R}^2)}).
\end{equation}
Using \eqref{2.3} we have $C(q) \|\eta_p \nabla \phi_p\|_{L^q(\Ra^2)} \leq \frac{1}{2} \| \nabla \phi_p \|_{L^q(\Ra^2)}$ so that from   \eqref{elliptique lol} we have
\begin{equation}
\label{tkkkkttt}
    || \nabla \phi_p||_{L^q(\mathbb{R}^2)} \leq C(q)  ||\eta_p ||_{L^q(\mathbb{R}^2)}.
\end{equation}
Next using Leibniz formula  we obtain
\begin{equation}
\begin{split}
 ||\partial^{\alpha} (\eta_p \nabla \phi_p) ||_{L^q(\mathbb{R}^2)} \leq C(q, \alpha) \bigg(&||\partial^{\alpha}  \eta_p||_{L^q(\mathbb{R}^2)}  || \nabla \phi_p ||_{L^{\infty}(\mathbb{R}^2)} \\&+
 \sum_{0 \leq \beta <\alpha} || \partial^{\beta} \eta_p ||_{L^{\infty}(\mathbb{R}^2)}  ||\partial^{\alpha-\beta}(\nabla \phi_p) ||_{L^q(\mathbb{R}^2)} \bigg).
 \end{split}
\end{equation}
Then, using \eqref{controleee} and \eqref{elliptique lol}, we have
\begin{equation}
\begin{split}
\label{3.17}
   ||\partial^{\alpha} (\nabla \phi_p)||_{L^q(\mathbb{R}^2)} \leq C(q, \alpha) \big(& || \partial^{\alpha} \eta_p ||_{L^q(\mathbb{R}^2)} \\ &+  \sum_{0 \leq \beta <\alpha} ||\partial^{\beta} \eta_p ||_{L^{\infty}(\mathbb{R}^2)} || \partial^{\alpha-\beta}(\nabla \phi_p) ||_{L^q(\mathbb{R}^2)} \big).
 \end{split}
 \end{equation}
We observe that (as a direct consequence of the rescaling \eqref{chgmt variable1}, \eqref{chgmt variable2}) there exists some positive constants $C_1(q,\alpha)$, $C_2(q,\alpha)$ and $C_3(q,\alpha)$, such that
\begin{align*}
    &||\partial^{\alpha}N_p||_{L^q(\mathbb{R}^2)} = \frac{C_1(q,\alpha)}{\varepsilon^{2+\alpha_1+2\alpha_2-\frac{3}{q}}_p} ||\partial^{\alpha}\eta_p||_{L^q(\mathbb{R}^2)},\\
    &||\partial^{\alpha}\partial_1\theta_p||_{L^q(\mathbb{R}^2)} = \frac{C_2(q,\alpha)}{\varepsilon^{2+\alpha_1+2\alpha_2-\frac{3}{q}}_p} ||\partial^{\alpha}\partial_1\phi_p||_{L^q(\mathbb{R}^2)},\\ &||\partial^{\alpha}\partial_2\theta_p||_{L^q(\mathbb{R}^2)} = \frac{C_3(q,\alpha)}{\varepsilon^{3+\alpha_1+2\alpha_2-\frac{3}{q}}_p} ||\partial^{\alpha}\partial_2\phi_p||_{L^q(\mathbb{R}^2)}.
\end{align*}
Combining with  \eqref{tkkkkttt} and \eqref{3.17}, we obtain \eqref{controle theta}  and \eqref{4.26}.
\end{proof}
\begin{rem}
Assuming we have bounds for $N_p$ in $W^{k,q}$ for any $k \in \mathbb{N}$, $1<q\leq \infty$ then, by induction using proposition \ref{lemme 4.6} and Sobolev embedding, we can bound $\theta_p$ in $W^{k,q}$ for any $k \in \mathbb{N}$,\;$1<q\leq \infty$.
\end{rem}
\section{Convolution equation}
\subsection{Reformulation}
The aim of this section is to rewrite  equation \eqref{E-K} as a convolution equation for $N_p$, similar to \eqref{SWW}. The first two lemmas are direct computations.
\begin{lemme}
We have
\begin{equation}
\label{camarche2}
    \partial_1N_p-\partial_1^2 \theta_p= \varepsilon_p^2 ( \mathbf{L}_1(N_p,\theta_p) +
    R(N_p,\theta_p) ),
\end{equation}
with
$$
\mathbf{L}_1(N_p,\theta_p)=\frac{1}{\varepsilon_p^2}(1-c_p)\partial_1 N_p + \partial_2^2\theta_p,
$$
$$
R(N_p,\theta_p)=- \gamma\partial_1 \left( N_p\partial_1\theta_p \right)- \gamma \varepsilon_p^2\partial_2(N_p\partial_2\theta_p).$$
\end{lemme}
\begin{lemme}
The function $N_p$ and $\theta_p$ satisfy

\begin{equation}
\label{4.13}
\begin{split}
   N_p(x)-c_p\partial_1(\theta_p(x)) =&+ \varepsilon_p^2 \bigg[\frac{\gamma g''(1)}{2} N^2_p-\varepsilon_p^2 \gamma^2 N_p^3l_1 \\&+ K\bigg(1-\gamma\varepsilon_p^2N_p(x)\bigg) \bigg( \frac{\partial_1^2 N_p}{K(1)}+\varepsilon_p^2\frac{\partial_2^2 N_p}{K(1)}\bigg)\\&+ \gamma \frac{(\partial_1 \theta_p)^2}{2}+\gamma\frac{\varepsilon_p^2(\partial_2 \theta_p)^2}{2}\\&-\gamma\frac{K'\bigg(1-\gamma\varepsilon_p^2N_p(x)\bigg)}{2K(1)}\big(\varepsilon_p^2\big(\partial_1N_p\big)^2+\varepsilon_p^4\big(\partial_2N_p)^2\big) \bigg],
\end{split}
\end{equation}
where $l_1$ is defined by
$$
g(\rho_p(x))=-\gamma\varepsilon_p^2N_p(x')+\varepsilon_p^4  \frac{g''(1)}{2}\gamma^2N^2_p(x')-\gamma^3\varepsilon_p^6 N_p^3(x')l_1(-\varepsilon^2_p\gamma N_p(x')).
$$
\end{lemme}
\begin{rem}
We recall that 
$$
x'=\bigg( \frac{\varepsilon_p x_1}{\sqrt{K(1)}}, \frac{\varepsilon_p^2x_2}{\sqrt{K(1)}} \bigg).
$$
In view of \eqref{2.3},  the function $l_1(-\varepsilon_p^2\gamma N_p(x'))=l_1(\rho_{p}(x)-1)$ is smooth and bounded in $L^{\infty}$ independently of p. To simplify we write
$$
g(\rho_p(x))=-\gamma\varepsilon_p^2N_p(x')+\varepsilon_p^4  \frac{g''(1)}{2}\gamma^2N^2_p(x')-\gamma^3\varepsilon_p^6 N_p^3(x')l_1(x').
$$
\end{rem}
Now we combine these two lemmas to obtain a Kadomtsev-Petviashvili equation with some remainder.
\begin{proposition} 
\label{proposition 5.3}%Quitte à faire le changement de variable  $N_p \rightarrow -a N_p  $, $\theta_p \rightarrow -a\theta_p$
\begin{equation}
\begin{split}
\label{presque KP1}
    \partial_1^4N_p- \Delta N_p-\mathbf{L}(N_p,\theta_p)= &-\partial_1^2 f_p\\&-\varepsilon_p^2\sum_{i+j=2}\partial_1^i\partial_2^j R_{\varepsilon_p}^{i,j},
    \end{split}
\end{equation}
with
\begin{equation}
\label{definition L}
    \mathbf{L}(N_p,\theta_p) =- (2\varepsilon_p^2 \partial_1^2\partial_2^2N_p +\varepsilon_p^4\partial_2^4N_p),
\end{equation}
\begin{equation}
\label{4.16}
\begin{split}
    \frac{1}{\gamma} R^{2,0}_{\varepsilon_p}(N_p,\theta_p)=&\bigg(\frac{c_p-1}{\varepsilon_p^2} N_p\partial_1\theta_p-\gamma N_p^3l_1+\frac{1}{2}(\partial_2 \theta_p)^2-\frac{K'(\rho_p)}{K(1)}(\partial_1N_p)^2\bigg) \\ &-\varepsilon_p^2\bigg(
    \frac{K'(\rho_p)}{K(1)}\big(\partial_2N_p)^2 \bigg) -\frac{h_1}{\gamma\varepsilon_p^4}h_2 ,
    \end{split}
\end{equation}
\begin{equation}
\label{4.17}
    \begin{split}
        \frac{1}{\gamma} R_{\varepsilon_p}^{0,2}(N_p,\theta_p)=&\bigg(\frac{g''(1)}{2}N^2_p+\frac{\big(\partial_1\theta_p\big)^2}{2} \bigg)+\varepsilon_p^2 \Bigg[-\gamma N_p^3l_1+\frac{(\partial_2\theta_p)^2}{2}-\frac{K'(\rho_p)(\partial_1N_p)^2}{2K(1)} \bigg]\\&-\varepsilon_p^4\frac{K'(\rho_p)(\partial_2N_p)^2}{2K(1)}-\frac{1}{\gamma\varepsilon_p^2}h_1h_2,
    \end{split}
\end{equation}
\begin{equation}
\label{4.18}
    \frac{1}{\gamma  } R_{\varepsilon_p}^{1,1}(N_p,\theta_p)= c_p N_p \partial_2\theta_p,
\end{equation}
\begin{equation}
    f_p=\gamma \bigg( N_p\partial_1\theta_p+\frac{g''(1)}{2}N^2_p+\frac{1}{2}(\partial_1\theta_p)^2\bigg).
\end{equation}
    Where $h_1$ and $h_2$ are some functions defined later.
\end{proposition}
\begin{rem}
We observe that
\begin{equation}
    \gamma\bigg( 1+\frac{g''(1)}{2}+ \frac{1}{2}\bigg)=\gamma\bigg(\frac{3+g''(1)}{2}\bigg)=\frac{1}{2}.
\end{equation}
Then passing, formally in \eqref{presque KP1}, to the limit in $p$ we have
\begin{equation}
    \partial_1^4N_0- \Delta N_0 +\frac{1}{2} \partial_1^2(N_0^2)=0,
\end{equation}
i.e. $N_0$ is a solution of \eqref{KP1}.
\end{rem}
\begin{proof}
Let 
\begin{equation}
\label{alpha}
\begin{split}
   \frac{1}{K(1)} h_1(x)=\frac{1}{K\big(1 -\gamma \varepsilon_p^2 N_p(x)\big)}-\frac{1}{K(1)}=h_3(-\varepsilon^2_p\gamma N_p(x))\varepsilon_p^2N_p(x),
\end{split}
\end{equation}
and 
\begin{equation}
\label{beta}
\begin{split}
h_2= N_p-c_p\partial_1\theta_p - \varepsilon_p^2 \bigg[&\gamma\frac{g''(1)}{2} N^2_p-\gamma^2\varepsilon_p^2 N_p^3l  + \frac{\gamma(\partial_1 \theta_p)^2}{2}+\gamma\frac{\varepsilon_p^2(\partial_2 \theta_p)^2}{2}\\&-\gamma\frac{K'\bigg(1-\varepsilon_p^2N_p\bigg)}{2K(1)}\big(\varepsilon_p^2\big(\partial_1N_p\big)^2+\varepsilon_p^4\big(\partial_2N_p)^2\big) \bigg].
\end{split}
\end{equation}
First of all, multiplying \eqref{4.13}   by  $K(1)\bigg(\frac{1}{K(1)}+\frac{1}{K\big(1-\gamma\varepsilon_p^2N_p\big)}-\frac{1}{K(1)}\bigg)$ we obtain
\begin{equation}
\label{camarche1}
\begin{split}
 N_p-c_p\partial_1\theta_p =+ \varepsilon_p^2 \bigg[&\gamma\frac{g''(1)}{2} N^2_p-\gamma^2\varepsilon_p^2 N_p^3l +  \bigg( \partial_1^2 N_p+\varepsilon_p^2\partial_2^2 N_p\bigg)\\&+\gamma \frac{(\partial_1 \theta_p)^2}{2}+\gamma\frac{\varepsilon_p^2(\partial_2 \theta_p)^2}{2}\\&-\gamma\frac{K'\bigg(1-\gamma\varepsilon_p^2N_p\bigg)}{2K(1)}\big(\varepsilon_p^2\big(\partial_1N_p\big)^2+\varepsilon_p^4\big(\partial_2N_p)^2\big) \bigg]\\& - h_1h_2.
\end{split}
\end{equation}
We have computing $-(\partial_1^2+\varepsilon_p^2\partial_2^2)(\ref{camarche1})+c\partial_1(\ref{camarche2})$
\begin{equation}
\begin{split}
    \partial_1^4N_p- \Delta N_p=&-\partial_1^2\bigg[\gamma \bigg( N_p\partial_1\theta_p+\frac{g''(1)}{2}N^2_p+\frac{1}{2}(\partial_1\theta_p)^2\bigg)\bigg]\\&+\mathbf{L}(N_p,\theta_p)-\varepsilon_p^2\sum_{i+j=2}\partial_1^i\partial_2^j R_{\varepsilon_p}^{i,j},
    \end{split}
\end{equation}
 this completes the proof of Claim \eqref{presque KP1}.
\end{proof}
We now recast \eqref{presque KP1} as a convolution equation.
\begin{proposition}
Let 
\begin{equation}
\label{bientotfini}
  \widehat{K_{\varepsilon_p}^{i,j}}(\xi)=\frac{\xi_1^i \xi_2^j}{ |\xi|^2+\xi_1^4+ 2\varepsilon_p^2\xi_1^2\xi_2^2+ \varepsilon_p^4 \xi_2^4},
\end{equation}
we have
\begin{equation}
\label{equationrecurrence}
N_p=K_{\varepsilon_p}^{2,0} \ast f_p + \sum_{i+j=2}  \varepsilon_p^2  K_{\varepsilon_p}^{i,j} \ast R_{\varepsilon_p}^{i,j}.
\end{equation}
\end{proposition}
\begin{proof}
It is a direct computation.
\end{proof}
Next we define
\begin{equation}
\label{Q}
    Q(\xi)= |\xi|^2+\xi_1^4+ 2\varepsilon_p^2\xi_1^2\xi_2^2+\varepsilon_p^4\xi_2^4.
\end{equation}
We need to etablish that the remainder term $R_{\varepsilon_p}$ are small enough in some Sobolev spaces.
\begin{proposition}
\label{controle reste}
There exists some positive constant $C$, not depending on $p$,  such that
\begin{equation}
\label{2}
    \begin{split}
       & \int_{\mathbb{R}^2} | R_{\varepsilon_p}^{1,1} |\leq C,\qquad \int_{\mathbb{R}^2} |f_p| \leq C,  
    \end{split}
\end{equation}
\begin{equation}
\label{33}
    \int_{\mathbb{R}^2} |R_{\varepsilon_p}^{0,2}| \leq C,\qquad \int_{\mathbb{R}^2} |R_{\varepsilon_p}^{2,0}| \leq \frac{C}{\varepsilon_p^{2}}.
\end{equation}
\end{proposition}
\begin{proof}
Proposition \ref{controle reste} is a direct consequence of \eqref{2.3}, identities in subsection \ref{pifpafpouf}, estimates in subsection \ref{pouuuuuuuuuuuuuuuuuuuuuuuuuuuf}, \eqref{controle de N  en normel L2} and Hölder's inequality.
For example we have
\begin{align*}
\int_{\Ra^2} \big(\partial_2\theta_p\big)^2= \frac{1}{K(1)\gamma^2\varepsilon_p^3}\int_{\Ra^2}(\partial_2 \phi_p)^2\leq     C,
\end{align*}
thus
\begin{align*}
    \int_{\mathbb{R}^2} | R_{\varepsilon_p}^{1,1} |\leq C \|N_p\|_{L^2(\mathbb{R}^2)}\|\p_2\theta_p\|_{L^2(\mathbb{R}^2)}\leq C.
\end{align*} Similarly we obtain the estimate on $f_p$ and \eqref{33}.  
\end{proof}
\subsection{Kernel estimates}
In order to use \eqref{equationrecurrence} we need to control $K_{\varepsilon_p}^{i,j}$. In this section we use computations of \cite{BGS2} section 5.
\begin{proposition}(\cite{BGS2} lemma 5.1)
\label{kernel controlle}
Let $0 \leq s <1$, there exists a constant $C(s)$ depending possibly on $s$, but not on $p$, such that
\begin{equation}
\label{kernerlss}
   ||K_{\varepsilon_p}^{2,0}||_{\dot{H}^{s}(\mathbb{R}^2)} \leq C(s) (1+ \varepsilon_p^{\frac{1}{2}-2s}), \, ||K_{\varepsilon_p}^{1,1}||_{\dot{H}^{s}(\mathbb{R}^2)} \leq C(s) (1+ \varepsilon_p^{-\frac{1}{2}-2s})  , 
\end{equation}
\begin{equation}
\label{kernelsss}
 ||K_{\varepsilon_p}^{0,2}||_{\dot{H}^{s}(\mathbb{R}^2)} \leq C(s) (1+ \varepsilon_p^{-\frac{3}{2}-2s}). 
 \end{equation}
  Thus we have 
 \begin{equation}
\label{kernels5.1}
    ||K_{\varepsilon_p}^{2,0}||_{H^{s}(\mathbb{R}^2)}+\varepsilon_p||K_{\varepsilon_p}^{1,1}||_{H^{s}(\mathbb{R}^2)}+\varepsilon_p^2||K_{\varepsilon_p}^{0,2}||_{H^{s}(\mathbb{R}^2)} \leq C(s),
\end{equation}
for any $0 \leq s\leq \frac{1}{4}.$
\end{proposition}

%\begin{rem}
%Hélas on a   pour i+j=3 
%\begin{equation}
%    ||\bighat{K}_{\varepsilon_p}^{i,j}||^2_{\dot{\mathbf{H}}^{\alpha}(\mathbb{R}^2)}= \infty
%\end{equation}
%\end{rem}
\begin{proposition} (\cite{BGS2} lemma 5.2) 
\label{inégalitéconvolution}
let $1<q<\infty$ and $0\leq i,j \leq 4$ integers such that $2 \leq i+j \leq 4$ we denote by 
$$
\kappa_{i,j} = max\left\{ i+2j-4,0\right\}. 
$$
Then  there exists a constant $C(q)$, depending possibly on $q$, but not on $p$ such that 
$$
|| K_{\varepsilon_p}^{i,j} *f ||_{L^q(\mathbb{R}^2)} \leq \frac{C(q)}{\varepsilon_p^{\kappa_{i,j}}} ||f||_{L^q(\mathbb{R}^2)},
$$
for any $f \in L^q(\mathbb{R}^2)$ and  $\varepsilon_p>0$.
\end{proposition}
\begin{rem}
As a consequence
$N_p$ and all its derivatives, belong to $L^q(\mathbb{R}^2)$ for any $1<q<2$. Indeed, we have 
\begin{equation}
N_p=K_{\varepsilon_p}^{2,0} \ast f_p + \sum_{i+j=2} \varepsilon_p^2  K_{\varepsilon_p}^{i,j} \ast R_{\varepsilon_p}^{i,j}.
\end{equation}
Thus, using lemma \ref{régularité} we obtain
$$
\p^{\alpha}N_p=K_{\varepsilon_p}^{2,0} \ast \p^{\alpha}f_p + \sum_{i+j=2} \varepsilon_p^2  K_{\varepsilon_p}^{i,j} \ast \p^{\alpha}R_{\varepsilon_p}^{i,j}.
$$
Then, combining the definition of $R_{\varepsilon_p}^{i,j}$, lemma \ref{régularité} with the proposition above, we have $\p^{\alpha}N_p \in L^q$ for any $1<q<2$. 
Therefore using \eqref{controle theta},  $\partial_1 \theta_p$ and $\partial_2 \theta_p$ belong to $L^q(\mathbb{R}^2)$ for any $1<q<2.$ Finally \eqref{4.13} and \eqref{camarche2} give,  respectively, $\p_1\theta_p \in W^{k,q}$ and $\p_2\theta_p \in W^{k,q}$ for any $k \in \mathbb{N}.$ 
\end{rem}
\section{Bounds in Sobolev spaces}
This section is devoted to the proof by induction of the following proposition:
\begin{proposition}
\label{proposition 6.1}
There exists $p_0>0$ such that for any $0\leq p \leq p_0$,  $\alpha \in \mathbb{N}^2$ and $1<q<\infty$, there exists $C(q, \alpha)$ such that
\begin{equation}
\begin{split}
\label{le truc à démontrer par recurrence}
    &||\partial^{\alpha} N_p||_{L^q(\mathbb{R}^2)}+||\partial_1\partial^{\alpha} N_p||_{L^q(\mathbb{R}^2)}+ ||\partial_2\partial^{\alpha} N_p||_{L^q(\mathbb{R}^2)}\\+&||\partial_1^2\partial^{\alpha} N_p||_{L^q(\mathbb{R}^2)}+\varepsilon_p||\partial_1\partial_2\partial^{\alpha} N_p||_{L^q(\mathbb{R}^2)}+\varepsilon_p^2||\partial_2^2\partial^{\alpha} N_p||_{L^q(\mathbb{R}^2)} \leq C(q, \alpha).
    \end{split}
\end{equation}
\end{proposition}
We have the following consequence. 
\begin{thm}\label{bornes sobolev}
There exists $p_0>0$ such that for any $0\leq p \leq p_0$, $k \in \mathbb{N}$ and $1<q\leq \infty$ there exists $C(k,q)$  such that
\begin{equation}
\label{11}
  ||N_p||_{W^{k,q}(\mathbb{R}^2)} +||\partial_1 \theta_p||_{W^{k,q}(\mathbb{R}^2)}+\varepsilon_p||\partial_2 \theta_p||_{W^{k,q}(\mathbb{R}^2)} \leq  C(k,q).
\end{equation}
\end{thm}
\begin{proof}[Proof of theorem \ref{bornes sobolev} assuming proposition \ref{proposition 6.1}]
By proposition \ref{proposition 6.1}, we have for any $k \in \mathbb{N},\; 1<q<\infty,$\, there exists a constant $C(k,q)$ not depending on $p$ such that
\begin{equation}
\label{6.22222}
     ||N_p||_{W^{k,q}(\mathbb{R}^2)} \leq C(k,q).
\end{equation}
Thus using Sobolev embedding \eqref{theorème de morrey}, we obtain 
\begin{equation}
     \forall k \in \mathbb{N},\; ||N_p||_{C^{k}_0(\mathbb{R}^2)} \leq C(k,q).
\end{equation}
By proposition \ref{lemme 4.6} we have
\begin{equation}
\label{6.4}
    \Gamma_p(q,\alpha) \leq C(q, \alpha) \bigg( 1 + \varepsilon_p^2 \sum_{0 \leq \beta <\alpha }\Gamma_p(q,\alpha-\beta) \bigg).
\end{equation}
Combining \eqref{controle theta} and \eqref{6.22222}
, the term
$\Gamma_p(q,(0,0))$ is  bounded independently of $p$. Then by induction and \eqref{6.4} the quantity $\Gamma_p(q,\alpha)$  is bounded independently of $p$, for any $1<q<\infty,$ $\alpha \in \mathbb{N}^2$. Finally, using  Sobolev embedding  \eqref{theorème de morrey} this result  is also true for $q=\infty$.
\end{proof}
\paragraph{Preliminary } We begin the proof of proposition \ref{proposition 6.1}. First of all, we have:
\begin{lemme}\label{wouaf wouaf}
For any $1 < q < \infty $, there exists a constant $C(q)$, independent of $p$, such that
\begin{equation}
\label{recurrence1}
    ||N_p(\partial_1 \theta_p)^2||_{L^q(\mathbb{R}^2)}+\varepsilon_p^2||N_p(\partial_2 \theta_p)^2||_{L^q(\mathbb{R}^2)} \leq C(q) ||N_p||_{L^{3q}(\mathbb{R}^2)}^3,
\end{equation}
and
\begin{equation}
\label{recurrence2}
\begin{split}
    &|| (\partial_1 \theta_p)^2 ||_{L^q(\mathbb{R}^2)} + \varepsilon_p ||N_p \partial_2 \theta_p ||_{L^q(\mathbb{R}^2)}+ \varepsilon_p^2 ||(\partial_2 \theta_p)^2||_{L^q(\mathbb{R}^2)}+ ||N_p \partial_1 \theta_p ||_{L^q(\mathbb{R}^2)} \\& \leq C(q)||N_p||_{L^{2q}(\mathbb{R}^2)}^2.
    \end{split}
\end{equation}
\end{lemme}
\begin{proof}
Hölder inequality and \eqref{controle theta} leads to
\begin{align*}
     &||N_p(\partial_1 \theta_p)^2||_{L^q(\mathbb{R}^2)}+\varepsilon_p^2||N_p(\partial_2 \theta_p)^2||_{L^q(\mathbb{R}^2)} \\& \leq  C(q) ||N_p||_{L^{3q}(\mathbb{R}^2)} \big( ||\partial_1 \theta_p ||^2_{L^{3q}(\mathbb{R}^2)}+ \varepsilon_p^2||\partial_2 \theta_p ||^2_{L^{3q}(\mathbb{R}^2)}
     \big)\\& \leq C(q) ||N_p||^3_{L^{3q}(\mathbb{R}^2)},
\end{align*}
a similar computations gives us  \eqref{recurrence2}.
\end{proof}
\begin{lemme} 
For any $1 < q < \infty $, there exists a constant $C(q)$, such that 
\label{lemme6.1}
\begin{equation}
\begin{split}
    \bigg|\bigg|\frac{h_1h_2}{\varepsilon_p^2}\bigg|\bigg|_{L^{q}(\mathbb{R}^2)} \leq  C(q) \bigg(& ||N_p||^2_{L^{2q}(\mathbb{R}^2)}+\varepsilon_p^2 ||N_p||^3_{L^{3q}(\mathbb{R}^2)} \\&+ \varepsilon_p^4 ||N_p||_{L^{\infty}(\mathbb{R}^2)} ||\partial_1N_p||^2_{L^{2q}(\mathbb{R}^2)}\\&+\varepsilon_p^6 ||N_p||_{L^{\infty}(\mathbb{R}^2)} ||\partial_2N_p||^2_{L^{2q}(\mathbb{R}^2)} \bigg),
    \end{split}
\end{equation}
\begin{equation}
   ||f_p||_{L^{q}(\mathbb{R}^2)}+\varepsilon_p ||R_{\varepsilon_p}^{1,1}||_{L^{q}(\mathbb{R}^2)} \leq C(q) ||N_p||_{L^{2q}(\mathbb{R}^2)}^2,
\end{equation}
\begin{equation}
\begin{split}
    ||R_{\varepsilon_p}^{2,0}||_{L^{q}(\mathbb{R}^2)} \leq \frac{C(q)}{\varepsilon_p^2} \bigg(&||N_p||^2_{L^{2q}} +\varepsilon_p^2 ||N_p||^3_{L^{3q}} \\ &+\varepsilon_p^2 ||\partial_1N_p||^2_{L^{2q}}\\&+ \varepsilon_p^4 ||\partial_2N_p||^2_{L^{2q}}+ \bigg|\bigg|\frac{h_1(x)h_2(x)}{\varepsilon_p^2}\bigg|\bigg|_{L^{q}(\mathbb{R}^2)} \bigg),
\end{split}
\end{equation}
\begin{equation}
\begin{split}
       ||R_{\varepsilon_p}^{0,2}||_{L^{q}(\mathbb{R}^2)} \leq C(q)\bigg( &||N_p||^2_{L^{2q}(\mathbb{R}^2)} +\varepsilon_p^2 ||N_p||^3_{L^{3q}(\mathbb{R}^2)}\\&+ \varepsilon_p^2 ||\partial_1N_p||^2_{L^{2q}(\mathbb{R}^2)}\\&+ \varepsilon_p^4 ||\partial_2N_p||^2_{L^{2q}(\mathbb{R}^2)} +\bigg|\bigg|\frac{h_1(x)h_2(x)}{\varepsilon_p^2}\bigg|\bigg|_{L^{q}(\mathbb{R}^2)} \bigg).
      \end{split}
\end{equation}
Where $h_1,\, h_2$ are defined in proposition \ref{proposition 5.3}.
\end{lemme}
\begin{proof}
The proof is a direct computation combined with lemma \ref{wouaf wouaf}, estimate \eqref{2.3} and observing that
\begin{align*}
\bigg|\bigg|\frac{(h_1h_2)}{\varepsilon_p^2}\bigg|\bigg|_{L^q(\Ra^2)}
\leq C(q) \bigg(
&\|N^2_p\|_{L^q(\Ra^2)}
+\|N_p\p_1 \theta_p\|_{L^q(\Ra^2)}\\
&+\varepsilon_p^2\|N_p^3\|_{L^q(\Ra^2)}
+\varepsilon_p^2\|N_p(\p_1\theta_p)^2\|_{L^q(\Ra^2)}
\\&+
\varepsilon_p^4\|N_p^4\|_{L^q(\Ra^2)}
+\varepsilon_p^4\|(\p_2 \theta_p)^2N_p \|_{L^q(\Ra^2)}
\\&+\varepsilon_p^4\|(\p_1 N_p)^2N_p  \|_{L^q(\Ra^2)}\\&+\varepsilon_p^6\|(\p_2 N_p)^2N_p  \|_{L^q(\Ra^2)}\bigg).
\end{align*}
\end{proof}
\begin{lemme}
\label{6.2}
For any $1< q <\infty$, there exists a constant $C(q)$, such that
\begin{equation}
\begin{split}
    ||N_p ||_{L^{q}(\mathbb{R}^2)}&+||\partial_1N_p ||_{L^{q}(\mathbb{R}^2)}+||\partial_2N_p ||_{L^{q}(\mathbb{R}^2)}\\+||\partial_1^2N_p ||_{L^{q}(\mathbb{R}^2)}&+\varepsilon_p||\partial_1 \partial_2N_p ||_{L^{q}(\mathbb{R}^2)} + \varepsilon_p^2 ||\partial_2^2N_p ||_{L^{q}(\mathbb{R}^2)}\\ \leq C(q)
    \bigg(&||N_p ||^2_{L^{2q}(\mathbb{R}^2)}+\varepsilon_p^2||N_p ||^3_{L^{3q}(\mathbb{R}^2)}+ \varepsilon_p^2||\partial_1N_p ||^2_{L^{2q}(\mathbb{R}^2)}\\&+ \varepsilon_p^4||\partial_2N_p ||^2_{L^{2q}(\mathbb{R}^2)}+ \varepsilon_p^4||\partial_1N_p ||^2_{L^{2q}(\mathbb{R}^2)}||N_p ||_{L^{\infty}(\mathbb{R}^2)}\\&+\varepsilon_p^6||\partial_2N_p ||^2_{L^{2q}(\mathbb{R}^2)}||N_p ||_{L^{\infty}(\mathbb{R}^2)} \bigg).
    \end{split}
\end{equation}
\end{lemme}
\begin{proof}
First we observe that
\begin{equation}
    \partial^\alpha K_{\varepsilon_p}^{j,k}=\imath^{\alpha_1+ \alpha_2} K_{\varepsilon_p}^{j+\alpha_1,k+\alpha_2},
\end{equation}
for any $|\alpha| \leq 2.$ Applying the operator $\p^{\alpha}$ to \eqref{equationrecurrence} we obtain 
\begin{equation}
    ||\partial^{\alpha}N_p ||_{L^{q}(\mathbb{R}^2)} \leq || \partial^{\alpha} K_{\varepsilon_p}^{2,0 } \ast f_p ||_{L^{q}(\mathbb{R}^2)}+ \varepsilon_p^2 \sum_{i+j=2} ||\partial^\alpha K_{\varepsilon_p}^{i,j} \ast R_{\varepsilon_p}^{i,j} ||_{L^{q}(\mathbb{R}^2)},
\end{equation}
then, with proposition \ref{inégalitéconvolution}, we deduce 
\begin{align*}
    ||N_p||_{L^{q}(\mathbb{R}^2)} &\leq ||K_{\varepsilon_p}^{2,0} \ast f_p||_{L^{q}(\mathbb{R}^2)}+ \sum_{i+j=2}\varepsilon_p^2|| K_{\varepsilon_p}^{i,j} \ast R_{\varepsilon_p}^{i,j} ||_{L^{q}(\mathbb{R}^2)} \\
    &\leq C(q) \bigg( ||f_p||_{L^{q}(\mathbb{R}^2)} + \varepsilon_p^2 \sum_{i+j=2} ||R_{\varepsilon_p}^{i,j}||_{L^{q}(\mathbb{R}^2)} \bigg),
\end{align*}
\begin{align*}
    ||\partial_1 N_p||_{L^{q}(\mathbb{R}^2)} &\leq || \partial_1 K_{\varepsilon_p}^{2,0} \ast f_p||_{L^{q}(\mathbb{R}^2)}+ \sum_{i+j=2}\varepsilon_p^2|| \partial_1K_{\varepsilon_p}^{i,j} \ast R_{\varepsilon_p}^{i,j} ||_{L^{q}(\mathbb{R}^2)} \\
    &\leq C(q) \bigg( ||f_p||_{L^{q}(\mathbb{R}^2)} + \varepsilon_p^2 \big( ||R_{\varepsilon_p}^{2,0} ||_{L^{q}(\mathbb{R}^2)}+||R_{\varepsilon_p}^{1,1} ||_{L^{q}(\mathbb{R}^2)}\big) \\&\qquad \qquad \qquad \qquad \qquad \qquad \qquad  \quad + \varepsilon_p||R_{\varepsilon_p}^{0,2} ||_{L^{q}(\mathbb{R}^2)}  \bigg).
\end{align*}
Similar estimates hold for $\partial_2N_p,\, \partial_1^2N_p,\, \partial_1 \partial_2N_p\ \text{and} \,\partial_2^2N_p $.\,
Finally using  lemma
\ref{lemme6.1}, we obtain lemma \ref{6.2}.
\end{proof}
\begin{lemme}
\label{6.3}
Let $2 \leq q \leq \infty$, there exists a positive constant $C(q)$, depending possibly on $q$, but not on p, such that 
\begin{equation}
\label{estimationdérivé}
   ||\partial_1 N_p||_{L^{q}(\mathbb{R}^2)}+\varepsilon_p ||\partial_2 N_p||_{L^{q}(\mathbb{R}^2)} \leq C(q) \varepsilon_p^{\frac{6}{q}-3},
\end{equation}
and
 \begin{equation}
\label{6.12}
    ||N_p||_{L^{q}(\mathbb{R}^2)} \leq C(q), \, \forall 2 \leq q \leq \frac{8}{3}.
\end{equation}
Moreover, for $\frac{8}{3}<q<8 $, there exists $C(q)$ such that the following
bound holds.
\begin{equation}
\label{6.13}
    \varepsilon_p^{\frac{2}{3}} ||N_p||_{L^{q}(\mathbb{R}^2)} \leq C(q).
\end{equation}
\end{lemme}
\begin{proof}
Using \eqref{controleee}, we have  
\begin{equation}
    ||\partial_1 N_p||_{L^{\infty}(\mathbb{R}^2)}\leq \frac{C}{\varepsilon_p^3}, \qquad {\rm and} \qquad ||\partial_2 N_p||_{L^{\infty}(\mathbb{R}^2)}\leq \frac{C}{\varepsilon_p^4}.
\end{equation}
 Combining with \eqref{cool} and \eqref{coool} we obtain \eqref{estimationdérivé} using standard interpolations inequalities. Applying  \eqref{jenaimarrre} on \eqref{equationrecurrence}  for any $0\leq s\leq \frac{1}{4}$  we obtain
\begin{equation}
\begin{split}
||N_p||_{H^{s}(\mathbb{R}^2)} \leq& ||K_{\varepsilon_p}^{2,0}||_{H^{s}(\mathbb{R}^2)} \big(||f_p||_{L^{1}(\mathbb{R}^2)}+\varepsilon_p^2||R_{\varepsilon_p}^{2,0}||_{L^{1}(\mathbb{R}^2)}\big)
\\&+\varepsilon_p^2||R_{\varepsilon_p}^{1,1}||_{L^{1}(\mathbb{R}^2)}||K_{\varepsilon_p}^{1,1}||_{H^{s}(\mathbb{R}^2)}\\&+\varepsilon_p^2||R_{\varepsilon_p}^{0,2}||_{L^{1}(\mathbb{R}^2)}||K_{\varepsilon_p}^{0,2}||_{H^{s}(\mathbb{R}^2)}.
\end{split}
\end{equation}
By proposition \ref{controle reste}, lemma \ref{lemme6.1} and claim \eqref{kernels5.1}, we deduce
\begin{equation}
    ||N_p||_{H^{s}(\mathbb{R}^2)} \leq C(s).
\end{equation}
Thus, by Sobolev embedding \eqref{rappelsuite}  we have
%comme $$
%H^{\alpha}(\mathbf{\Ra^2}) \hookrightarrow L^q(\mathbb{R}^2), \, \forall   2 \leq q \leq \frac{2}{1-\alpha}
%$$ 
\begin{equation}
    ||N_p||_{L^{q}(\mathbb{R}^2)} \leq C(q), \, \forall 2 \leq q \leq \frac{8}{3}.
\end{equation}
Let
 $\nu>0$ combining \eqref{6.12} and \eqref{estimationdérivé}  there exists a constant $C(\nu)$ and $q>2$ such that 
$$
||N_p||_{W^{1,q}(\mathbb{R}^2)} \leq C(\nu) \bigg( 1 + \varepsilon_p^{-1-\nu} \bigg).
$$
 Thus, by Sobolev embedding \eqref{holder espace injection}, we obtain 
$$
||N_p||_{L^{\infty}(\mathbb{R}^2)} \leq C(\nu) \bigg( 1 + \varepsilon_p^{-1-\nu} \bigg).
$$
Combining with \eqref{6.12}  we have \eqref{6.13} by interpolation between Lebesgue spaces.
\end{proof}
We are now able to prove the first step of our induction.
\paragraph{The first step}
Let $1<q<\infty$, there exists a constant $C(q)$, such that
\begin{equation}
\begin{split}
\label{j'ai obtenu ça}
        ||N_p ||_{L^{q}(\mathbb{R}^2)}+||\partial_1N_p ||_{L^{q}(\mathbb{R}^2)}+||\partial_2N_p ||_{L^{q}(\mathbb{R}^2)}+\\||\partial_1^2N_p ||_{L^{q}(\mathbb{R}^2)}+\varepsilon_p||\partial_1 \partial_2N_p ||_{L^{q}(\mathbb{R}^2)} + \varepsilon_p^2 ||\partial_2^2N_p ||_{L^{q}(\mathbb{R}^2)} \leq C(q).
        \end{split}
\end{equation}
\begin{proof}
First of all for any $1 <q\leq \frac{4}{3}$ we have, using lemma \ref{6.2}, inequalities  \eqref{6.13}, \eqref{estimationdérivé}  and  \eqref{6.12},
\begin{equation}
\begin{split}
   ||N_p ||_{L^{q}(\mathbb{R}^2)}+||\partial_1N_p ||_{L^{q}(\mathbb{R}^2)}+||\partial_2N_p ||_{L^{q}(\mathbb{R}^2)}+&\\||\partial_1^2N_p ||_{L^{q}(\mathbb{R}^2)}+\varepsilon_p||\partial_1 \partial_2N_p ||_{L^{q}(\mathbb{R}^2)} + \varepsilon_p^2 ||\partial_2^2N_p ||_{L^{q}(\mathbb{R}^2)}  &\leq C(q).
\end{split}
\end{equation}
Then, by Sobolev embedding \eqref{embedding sobolev} and standard interpolations
inequalities,
\begin{equation}
    ||N_p ||_{L^{q}(\mathbb{R}^2)}+\varepsilon_p||\partial_1N_p ||_{L^{q}(\mathbb{R}^2)}+\varepsilon_p^2||\partial_2N_p ||_{L^{q}(\mathbb{R}^2)} \leq C(q),
\end{equation}
for any $1 < q \leq 4$.
Then, for any $1<q \leq 2$ by lemma \ref{6.2}, we obtain
\begin{equation}
\begin{split}
 ||N_p ||_{L^{q}(\mathbb{R}^2)}+||\partial_1N_p ||_{L^{q}(\mathbb{R}^2)}+||\partial_2N_p ||_{L^{q}(\mathbb{R}^2)}+&\\||\partial_1^2N_p ||_{L^{q}(\mathbb{R}^2)}+\varepsilon_p||\partial_1 \partial_2N_p ||_{L^{q}(\mathbb{R}^2)} + \varepsilon_p^2 ||\partial_2^2N_p ||_{L^{q}(\mathbb{R}^2)}  &\leq C(q).   
 \end{split}
\end{equation}
We deduce, by Sobolev embedding \eqref{embedding sobolev}, that
\begin{equation}
     ||N_p ||_{L^{q}(\mathbb{R}^2)}+\varepsilon_p||\partial_1N_p ||_{L^{q}(\mathbb{R}^2)}+\varepsilon_p^2||\partial_2N_p ||_{L^{q}(\mathbb{R}^2)} \leq C(q) ,
     \end{equation}
$\forall$ $1<q<\infty. $
Finally using lemma $\ref{6.2}$ we obtain claim \eqref{j'ai obtenu ça}.
\end{proof}
\paragraph{Inductive step }
We fix $k \in \mathbb{N}$.
In this part we assume that \eqref{le truc à démontrer par recurrence} is true for any  $1<q<\infty$ and any $\alpha \in \mathbb{N}^2$ such that $|\alpha| \leq k. $  We will prove that \eqref{le truc à démontrer par recurrence} holds for any $|\alpha|\leq k+1$. 
\begin{lemme}
There exists $p_0>0$ such that for any $0\leq p \leq p_0$,  $1<q<\infty$ and $\alpha \in \mathbb{N}^2$ such that $|\alpha| \leq k+1$, then there exists a constant
 $C(q, \alpha)$  such that
\begin{equation}
\label{6.18}
    ||\partial^{\alpha}\partial_1 \theta_p||_{L^{q}(\mathbb{R}^2)}+\varepsilon_p||\partial^{\alpha}\partial_2 \theta_p||_{L^{q}(\mathbb{R}^2)} \leq C(q, \alpha).
\end{equation}
\end{lemme}
\begin{proof}
Let $\alpha \in \mathbb{N}^2$, $|\alpha| \leq k+1.$
Using Sobolev embedding \eqref{injection ck sobolev} and our hypothesis, we have  
\begin{equation}
    ||N_p||_{C_0^k(\mathbb{R}^2)} \leq C(k),
\end{equation}
where $C(k)$ is a constant not depending on $p$. Then, by \eqref{4.26}, we obtain
\begin{equation}
\begin{split}
   & ||\partial^{\alpha}\partial_1 \theta_p||_{L^{q}(\mathbb{R}^2)}+\varepsilon_p||\partial^{\alpha}\partial_2 \theta_p||_{L^{q}(\mathbb{R}^2)} \leq C(q, \alpha) \bigg( ||\partial^{\alpha}N_p||_{L^{q}(\mathbb{R}^2)}\\&+ \varepsilon_p^2 \sum_{0\leq \beta <\alpha } \big( 
    ||\partial^{\alpha-\beta }\partial_1 \theta_p||_{L^{q}(\mathbb{R}^2)}+\varepsilon_p||\partial^{\alpha-\beta}\partial_2 \theta_p||_{L^{q}(\mathbb{R}^2)}\big) \bigg).
    \end{split}
\end{equation}
If we denote 
$$
S^q_k=\sum_{|\alpha|\leq k+1}\bigg(||\partial^{\alpha}\partial_1 \theta_p||_{L^{q}(\mathbb{R}^2)}+\varepsilon_p||\partial^{\alpha}\partial_2 \theta_p||_{L^{q}(\mathbb{R}^2)} \bigg),
$$
summing the previous inequalities, we obtain
$$
S_k^q \leq C(q, \alpha) \bigg( \varepsilon_p^2 S_k^q+ \sum_{|\alpha|\leq k+1}||\partial^{\alpha}N_p||_{L^{q}(\mathbb{R}^2)} \bigg).
$$
We conclude by using \eqref{le truc à démontrer par recurrence}.  
\end{proof}
\begin{lemme}\label{lemma 7.7}
There exists a positive constant $C(q, \alpha)$ depending possibly on q and $\alpha$, but not on p, such that 
\begin{equation}
\begin{split}
\label{6.19}
    & \qquad \qquad ||\partial^{\alpha} f_p ||_{L^{q}(\mathbb{R}^2)}+||\partial^{\alpha} R_{\varepsilon_p}^{0,2} ||_{L^{q}(\mathbb{R}^2)}\\&+\varepsilon_p||\partial^{\alpha} R_{\varepsilon_p}^{1,1} ||_{L^{q}(\mathbb{R}^2)}+\varepsilon_p^2||\partial^{\alpha} R_{\varepsilon_p}^{2,0} ||_{L^{q}(\mathbb{R}^2)}\leq C(q, \alpha),
\end{split}
\end{equation}
for any $1\leq q<\infty$ , $| \alpha|\leq k+1$.
\end{lemme}
\begin{proof}
We will detail the computation for $h_1h_2$ defined in proposition \ref{proposition 5.3}.
 First of all we have using \eqref{alpha}, \eqref{beta}
\begin{align*}
\bigg|\bigg|\frac{\p^{\alpha}(h_1h_2)}{\varepsilon_p^2}\bigg|\bigg|_{L^q(\Ra^2)}
\leq C \bigg(
&\|\p^{\alpha}(\delta N^2_p)\|_{L^q(\Ra^2)}
+\|\p^{\alpha}(\delta N_p\p_1 \theta_p)\|_{L^q(\Ra^2)}\\
&+\varepsilon_p^2\|\p^{\alpha}(\delta N_p^3)\|_{L^q(\Ra^2)}
+\varepsilon_p^2\|\p^{\alpha}(N_p(\p_1\theta_p)^2\delta)\|_{L^q(\Ra^2)}
\\&+
\varepsilon_p^4\|\p^{\alpha}(N_p^4l\delta)\|_{L^q(\Ra^2)}
+\varepsilon_p^4\|\p^{\alpha}((\p_2 \theta_p)^2N_p \delta)\|_{L^q(\Ra^2)}
\\&+\varepsilon_p^4\|\p^{\alpha}((\p_1 N_p)^2N_p \delta
K')\|_{L^q(\Ra^2)}\\&+\varepsilon_p^6\|\p^{\alpha}((\p_2 N_p)^2N_p \delta K')\|_{L^q(\Ra^2)}\bigg),
\end{align*}
where $\delta=\delta(-\varepsilon^2_p\gamma N_p), \,K'=K'(-\varepsilon^2_p\gamma N_p), \, l=l(-\varepsilon^2_p\gamma N_p).$ Using our hypothesis and the chain rule we have for any $\beta \in \mathbb{N}^2, \, |\beta|\leq k+1,\,  1<q<\infty$
$$
\|\p^{\beta}(K'(-\varepsilon^2_p\gamma N_p))\|_{L^q(\Ra^2)}+\|\p^{\beta}(\delta(-\varepsilon^2_p\gamma N_p))\|_{L^q(\Ra^2)}+\|\p^{\beta}(l(-\varepsilon^2_p\gamma N_p))\|_{L^q(\Ra^2)} \leq C.
$$
Thus using Leibniz  formula, our hypothesis and \eqref{6.18} we have 
$$\bigg|\bigg|\frac{\p^{\alpha}(h_1h_2)}{\varepsilon^2}\bigg|\bigg|_{L^q(\Ra^2)}
\leq C.$$
Applying the operator $\p^{\alpha}$ on  \eqref{4.16}, \eqref{4.17}, \eqref{4.18} and using our hypothesis and \eqref{6.18} we obtain claim \eqref{6.19}.
\end{proof}
Finally let $|\alpha | \leq k+1$ using  \eqref{equationrecurrence}  we have 
\begin{equation}
    \partial^{\alpha}N_p=K_{\varepsilon_p}^{2,0} \ast \partial^{\alpha}f_p + \sum_{i+j=2} \varepsilon_p^2 K_{\varepsilon_p}^{i,j} \ast \partial^{\alpha}R_{\varepsilon_p}^{i,j}.
\end{equation}
Then using proposition \ref{inégalitéconvolution}, there exists $C(q)$ such that 
 \begin{equation}
    ||\partial^{\alpha}N_p||_{L^{q}(\mathbb{R}^2)}\leq C(q) \bigg( || \partial^{\alpha}f_p||_{L^{q}(\mathbb{R}^2)} + \sum_{i+j=2} \varepsilon_p^2 || \partial^{\alpha}R_{\varepsilon_p}^{i,j}||_{L^{q}(\mathbb{R}^2)}, \bigg),
\end{equation}
 \begin{equation}
 \begin{split}
 \label{6.22}
     ||\partial^{\alpha} \partial_1 N_p ||_{L^{q}(\mathbb{R}^2)} \leq C(q) \bigg(& ||\partial^{\alpha} f_p ||_{L^{q}(\mathbb{R}^2)}
     \\&+\varepsilon_p^2\big(||\partial^{\alpha} R_{\varepsilon_p}^{2,0} ||_{L^{q}(\mathbb{R}^2)}+||\partial^{\alpha} R_{\varepsilon_p}^{1,1} ||_{L^{q}(\mathbb{R}^2)} \big)\\&\qquad \qquad \qquad \qquad +\varepsilon_p ||\partial^{\alpha} R_{\varepsilon_p}^{0,2} ||_{L^{q}(\mathbb{R}^2)} \bigg),
    \end{split}
 \end{equation}
and
  \begin{equation}
  \label{6.23}
  \begin{split}
          ||\partial^{\alpha} \partial_2 N_p ||_{L^{q}(\mathbb{R}^2)}+ ||\partial^{\alpha} \partial_1^2 N_p ||_{L^{q}(\mathbb{R}^2)}+\varepsilon_p||\partial^{\alpha} \partial_1 \partial_2 N_p ||_{L^{q}(\mathbb{R}^2)}+\varepsilon_p^2||\partial^{\alpha}  \partial_2^2 N_p ||_{L^{q}(\mathbb{R}^2)}
          \\
          \leq C(q) \bigg( ||\partial^{\alpha} f_p ||_{L^{q}(\mathbb{R}^2)}
     +\varepsilon_p^2||\partial^{\alpha} R_{\varepsilon_p}^{2,0} ||_{L^{q}(\mathbb{R}^2)}+\varepsilon_p||\partial^{\alpha} R_{\varepsilon_p}^{1,1} ||_{L^{q}(\mathbb{R}^2)} +||\partial^{\alpha} R_{\varepsilon_p}^{0,2} ||_{L^{q}(\mathbb{R}^2)} \bigg).
  \end{split}
 \end{equation}
 We conclude the induction by using \eqref{6.19}.
 \paragraph{Conclusion}
 We have proved proposition \ref{proposition 6.1}, theorem \ref{bornes sobolev} follows from our discussion at the beginning of the section. 
 \section{Strong convergence}
 \label{patate}
 \subsection{Strong local convergence}
 \begin{proposition}
 \label{convergencefortelocal}
 Let $(p_n)_{n \in \mathbb{N}}$ such that $p_n \rightarrow 0$. Then there exists $N_0$ a non constant solution of \eqref{SW}  such that, up to a subsequence, 
 \begin{equation}
     \label{7.1}
        \forall 1<q<\infty,\, \forall k \in\mathbb{N}, \, N_{p_n} \rightharpoonup N_0 \,\; \mbox{in} \, \, W^{k,q}\big(\mathbb{R}^2\big), \, \mbox{when}\; n \rightarrow \infty.
 \end{equation}
Thus for any $0\leq \gamma <1$ and any compact subset $\mathcal{K}$, we obtain  
 \begin{equation}
 \label{7.2}
     N_{p_n} \rightarrow N_0 \; \mbox{in} \; C^{0,\gamma}(\mathcal{K}),\, \mbox{when }\, n \rightarrow \infty.
 \end{equation}
 \end{proposition}
 \begin{proof}
Combining \eqref{11} with Banach-Alaoglu theorem, there exists a subsequence $(p_n)_{n \in \mathbb{N}}$ and a function $N_0$ such that $N_{p_n} \rightharpoonup N_0 \; \text{in} \, W^{k,q}\big(\mathbb{R}^2\big)$  for any $1<q<\infty$, $k \in \mathbb{N}$. Then  \eqref{7.2} is a consequence of \eqref{7.1} and \eqref{injection compacte holder}. Thus $N_0$ is non constant using \eqref{minimal}. We will now prove that $N_0$ is a solution of \eqref{SW}. 
We recall that
     \begin{equation}
    f_p=\gamma \bigg( N_p\partial_1\theta_p+\frac{g''(1)}{2}N^2_p+\frac{1}{2}(\partial_1\theta_p)^2\bigg),
\end{equation}
with 
\begin{equation}
  \gamma\bigg( 1+\frac{g''(1)}{2}+ \frac{1}{2}\bigg)=\gamma\bigg(\frac{3+g''(1)}{2}\bigg)=\frac{1}{2},
\end{equation}
and
\begin{equation}
\begin{split}
\label{poufpoufpouf}
    \partial_1^4N_p- \Delta N_p= &-\partial_1^2f_p+\mathbf{L}(N_p,\theta_p)\\&-\varepsilon_p^2\sum_{i+j=2}\partial_1^i\partial_2^j R_{\varepsilon_p}^{i,j}.
    \end{split}
\end{equation}
Using \eqref{definition L} and theorem \ref{bornes sobolev} we have
$$
\|\mathbf{L}(N_{p},\theta_{p})\|_{L^2(\Ra^2)} \underset{n \rightarrow \infty}{\longrightarrow} 0.
$$
Using \eqref{4.17} and theorem \ref{bornes sobolev} we have 
$$
\varepsilon_p^2||\p_1^2R_{\varepsilon_p}^{2,0}||_{L^{1}(\mathbb{R}^2)}\underset{n \rightarrow \infty}{\longrightarrow} 0,
$$
then using \eqref{6.19} we have 
$$
\varepsilon_p^2||\p_2^2R_{\varepsilon_p}^{0,2}||_{L^{1}(\mathbb{R}^2)}+\varepsilon_p^2||\p_1\p_2R_{\varepsilon_p}^{1,1}||_{L^{1}(\mathbb{R}^2)}\underset{n \rightarrow \infty}{\longrightarrow} 0.
$$
On the other hand, using \eqref{7.2}, we have 
$$
\|N_{p_n}^2-N_0^2\|_{L^{\infty}(K)}\underset{n \rightarrow \infty}{\longrightarrow} 0,
$$
for any compact $\mathcal{K}$.
Thus $$N_{p_n}^2\underset{n \rightarrow \infty}{\longrightarrow} N_0^2\; \text{in} \; D'(\Ra^2).$$
 Combining \eqref{Ntheta} and $$ ||(\partial_1 \theta_p)^2- N^2_p||_{L^1(\mathbb{R}^2)} \leq ||\partial_1 \theta_p- N_p||_{L^2(\mathbb{R}^2)}||\partial_1 \theta_p+ N_p||_{L^2(\mathbb{R}^2)}, $$
we obtain 
$$
||N\partial_1 \theta_p- N^2_p||_{L^1(\mathbb{R}^2)} \underset{p \rightarrow 0}{\longrightarrow} 0, \, ||(\partial_1 \theta_p)^2- N^2_p||_{L^1(\mathbb{R}^2)}\underset{p \rightarrow 0}{\longrightarrow} 0.
$$
Thus, we deduce
\begin{align*}
    \left|\left|\frac{1}{2}N_p^2-f_p\right|\right|_{L^{1}(\mathbb{R}^2)} \leq &\big(  |\gamma| ||N_p\partial_1 \theta_p- N^2_p||_{L^1(\mathbb{R}^2)}\\&+ \frac{|\gamma|}{2}||(\partial_1 \theta_p)^2- N^2_p||_{L^1(\mathbb{R}^2)}\bigg) \underset{p \rightarrow 0}{\longrightarrow}0.
\end{align*}
Finally passing to the limit in \eqref{poufpoufpouf} we have 
\begin{equation}
    \partial_1^4N_0- \Delta N_0 +\frac{1}{2} \partial_1^2(N_0^2)=0.
\end{equation}
\end{proof}

We will now prove the strong global convergence.
\subsection{Strong global convergence}
We recall that
\begin{align*}
E(\rho_p,\phi_p)=\frac{K(1)\gamma^2\varepsilon_p}{2} \bigg(E_0(N_p, \theta_p) +\varepsilon_p^2 (E_2(N_p,\theta_p)))+
  \varepsilon_p^4 (E_4(N_p,\theta_p)) \bigg),
\end{align*}
with
$$E_0(N_p, \theta_p)=\int_{\mathbb{R}^2}  N^2_p+\partial_1 \theta_p^2, $$

\begin{align*}
E_2(N_p, \theta_p)=2 \bigg(\int_{\mathbb{R}^2}    \frac{(\partial_1 N_p(x))^2}{2}+\frac{|\partial_2 \theta_p (x)|^2}{2}- \frac{\gamma}{6}\big( 3N_p(\partial_1 \theta_p)^2+ g''(1)N_p^3)dx\bigg),
\end{align*}

\begin{align*}
E_4(N_p,\theta_p)=\int_{\mathbb{R}^2}& K\bigg(1-\gamma\varepsilon_p^2N_p (x) \bigg) \frac{1}{\sqrt{K(1)}^2}|\partial_2N_p(x)|^2-
\gamma N_p(x)|\partial_2 \theta_p (x)|^2\\&+N_p(x) j(x)\bigg( 
 (\partial_1N_p(x))^2 +  \varepsilon_p^2 (\partial_2 N_p(x))^2 \bigg)\\&+\gamma^4  N_p^4(x)l(x) dx,
\end{align*}
and 
\begin{align*}
    E_{KP}(\p_1 \theta_p)=\frac{1}{2}\int_{\Ra^2}(\p_1^2 \theta_p)^2+\frac{1}{2}\int_{\Ra^2}(\p_2 \theta_p)^2-\frac{1}{6}\int_{\Ra^2}(\p_1 \theta_p)^3.
\end{align*}
\begin{proposition}
\label{proposition7.2}
Let $(p_n)_{n \in \mathbb{N}}$ a sequence which converges to 0 and satisfies \eqref{7.1} and \eqref{7.2}. Then, up to a subsequence, there exists a positive constant $\mu_0$ such that
\begin{equation}
\label{le truc de la proposition 8.2}
    E_{KP}(\partial_1 \theta_{p_n}) \rightarrow \mathcal{E}^{KP}(\mu_0), \quad \mbox{ and} \quad \int_{\Ra^2}|\partial_1 \theta_{p_n}|^2 \underset{n \rightarrow \infty}{\longrightarrow} \mu_0>0.
\end{equation}
See section 2 for the definition of $\mathcal{E}^{KP}$.
\end{proposition}
First of all, we prove three lemmas.
\begin{lemme}
\label{3}
There exists a constant $C$, not depending on p, such that 
\begin{equation}
    E(\rho_p,\phi_p)-p \leq  \frac{-1}{(K(1)\gamma^2)^2 54 S_{KP}^2}p^3+Cp^4.
\end{equation}
\end{lemme}
\begin{proof}
Let $\omega$ a ground state of \eqref{SW} and $p>0$. According to \cite{BouardSaut3,debouardsaut} $\omega$ is smooth, belongs to $L^q(\mathbb{R}^2)$ for any $q>1$  its first order derivatives are in $L^2(\mathbb{R}^2)$. There exists $v$ such that $\partial_1v=\omega$ (see lemma 3.9 \cite{BGS} and \cite{Gravejat4}) with $v$ smooth. Moreover $v$ is in $L^q(\mathbb{R}^2),\, q>2$ and its gradient belongs to $L^q(\mathbb{R}^2),\, q>1$.
%On prend v,w comme dans l'article \cite{Gravejat3}, \cite{BGS} avec $\partial_1v=w$ et dans ce cas 
We let
\begin{equation}
\begin{split}
     \rho(x_1,x_2)=1- \varepsilon^2 \gamma \omega \bigg(\varepsilon \frac{ x_1}{\sqrt{K(1)}}, \varepsilon^2\frac{x_2 }{\sqrt{K(1)}} \bigg),
 \end{split}
\end{equation}
\begin{equation}
    \phi(x_1,x_2)=- \gamma \sqrt{K(1)}\varepsilon v\bigg(\varepsilon \frac{ x_1}{\sqrt{K(1)}}, \varepsilon^2\frac{x_2 }{\sqrt{K(1)}} \bigg),
\end{equation}
where the constant  $\varepsilon$ is chosen such that
$$
p=K(1)\gamma^2\varepsilon\int_{\Ra^2} \omega^2.
$$
We observe that $\rho \in 1+H^1(\mathbb{R}^2)$ and $\phi \in \dot{H}^1(\mathbb{R}^2)$.
Thus, by computation, we have $$E(\rho,\phi)=K(1)\gamma^2 \varepsilon \int_{\Ra^2} \omega^2+K(1)\gamma^2
\varepsilon^3 E_{KP}(\omega)+\frac{K(1)\gamma^2}{2}\varepsilon^5E_4(\omega,v). $$
Then, using lemma \ref{lemme 2.1}, we obtain 
\begin{align*}
    E(\rho,\phi)-p&\leq K(1)\gamma^2\varepsilon^3E_{KP}(\omega)+C\varepsilon^5
    \\ &\leq K(1)\gamma^2\bigg(\frac{p}{K(1)\gamma^2\big(\int \omega^2\big)}\bigg)^3E_{KP}(\omega)+Cp^4
    \\ &\leq \frac{p^3}{(K(1)\gamma^2)^2}\frac{E_{KP}(\omega)}{(\int \omega^2)^3}+Cp^4
    \\ & \leq \frac{p^3}{(K(1)\gamma^2)^2}\bigg(\frac{-1}{54S_{KP}^2}\bigg)+Cp^4.
\end{align*}
For $p\ll 1$, since $\omega \in L^{\infty}(\Ra^2)$ we have $|\rho-1|\ll 1\; \text{and} \; |\rho_p-1|\ll 1$.
Thus using \eqref{minimisateur pour E-k}  we have
 \begin{align*}
    E(\rho_p,\phi_p)-p&=E(\rho_p,\phi_p)-p
    \\& \leq E(\rho,\phi)-p\\&= E(\rho,\phi)-p\\ & \leq \frac{p^3}{(K(1)\gamma^2)^2}\bigg(\frac{-1}{54S_{KP}^2}\bigg)+Cp^4.
 \end{align*}
\end{proof}
\begin{lemme}
\label{16}
We have
\begin{equation}
\label{8.11}
    E(N_p,\theta_p)-p=K(1)\gamma^2\varepsilon_p^3E_{Kp}\big(\partial_1\theta_p\big)+o(\varepsilon_p^3),
\end{equation}
and
\begin{equation}
    -\frac{1}{54S_{KP}^2}\bigg( \int_{\Ra^2} (\partial_1\theta_p)^2\bigg)^3 \leq E_{KP}(\partial_1\theta_p)\leq -\frac{1}{54S_{KP}^2}\bigg( \int_{\Ra^2} (\partial_1\theta_p)^2\bigg)^3 +o(1).
\end{equation}
\end{lemme}
\begin{proof}
Combining \eqref{4.13} and \eqref{7.1} we obtain
\begin{equation}
\label{meilleurconvergence}
    \frac{1}{\varepsilon_p^2}\int_{\mathbb{R}^2}(N_p-\partial_1\theta_p)^2 \underset{p \rightarrow 0}{\longrightarrow} 0,
\end{equation}
and \begin{equation}
\label{meilleurconvergence pour l'autre}
    ||\partial_1N_p-\partial_1^2\theta_p||_{L^2(\mathbb{R}^2)}\leq M \varepsilon_p.
\end{equation}
Then, we have
\begin{align*}
    E(N_p,\theta_p)-p=&\frac{K(1)\gamma^2\varepsilon_p}{2}\int_{\Ra^2}\big(N_p-\partial_1\theta_p\big)^2+K(1)\gamma^2\varepsilon_p^3E_{KP}(\partial_1 \theta_p)
    \\
    +&\frac{K(1)\gamma^2\varepsilon_p^3}{2}\bigg(\int_{\Ra^2}\big((\partial_1 N_p(x))^2-(\partial_1^2 \theta_p)^2)dx + \\&+\frac{\gamma g''(1)}{3}(-N_p^3+ (\partial_1\theta_p)^3)+ \gamma\big(-N_p(\partial_1\theta_p)^2+(\partial_1\theta_p)^3 )dx \bigg)\\&+
    \frac{K(1)\gamma^2\varepsilon_p^5}{2}E_4(N_p,\theta_p).
\end{align*}
Thus
\begin{equation}
    E(N_p,\theta_p)-p=K(1)\gamma^2\varepsilon_p^3E_{KP}(\partial_1 \theta_p)+o(\varepsilon_p^3).
\end{equation}
Using lemma \ref{lemme 2.1}
\begin{equation}
    E_{KP}(\partial_1\theta_p) \geq \mathcal{E}_{min}^{KP} \bigg(\int_{\Ra^2} (\partial_1\theta_p)^2 \bigg) =-\frac{1}{54S_{KP}^2}\bigg( \int_{\Ra^2} (\partial_1\theta_p)^2\bigg)^3.
\end{equation}
Then combining \eqref{8.11}, lemma \ref{3} and \eqref{meilleurconvergence} we obtain
\begin{align*}
    E_{KP}(\partial_1 \theta_p)&\leq -\frac{p^3}{(K(1)\gamma^2)^354 S^2_{KP}}\frac{1}{\varepsilon^3}+o(1)\\
&\leq \frac{-1}{54S^2_{KP}}\bigg(\int_{\Ra^2} N_p\partial_1 \theta_p\bigg)^3+o(1)\\
&\leq \frac{-1}{54S^2_{KP}}\bigg(\int_{\Ra^2} (\partial_1 \theta_p)^2\bigg)^3+o(1).
\end{align*}
This ends the proof of lemma \ref{16}. 
\end{proof}
We are now able to prove proposition  \ref{proposition7.2}.
\begin{proof}[Proof of proposition \ref{proposition7.2} ]
Using \eqref{7.1} and \eqref{meilleurconvergence} we have
\begin{equation}
    \liminf_{n\rightarrow \infty} \int_{\Ra^2} (\partial_1\theta_{p_n})^2 \geq \int_{\mathbb{R}^2} N_0^2,
\end{equation}
thus, up to extraction, 
\begin{equation}
\label{7.13}
    \lim_{n\rightarrow \infty} \int_{\Ra^2} (\partial_1\theta_{p_n})^2 \rightarrow \mu_0, \; \mbox{when}\; n \rightarrow \infty,
\end{equation}
where
\begin{equation}
    \mu_0 \geq \int_{\mathbb{R}^2} N_0^2 >0.
\end{equation}
Combining lemma \ref{16}, \eqref{7.13} and lemma \ref{lemme 2.1} we obtain \begin{equation}
    E_{KP}(\partial_1 \theta_{p_n}) \rightarrow \mathcal{E}^{KP}_{min}(\mu_0), \quad {\rm and} \quad \int_{\Ra^2}|\partial_1 \theta_{p_n}|^2 \rightarrow \mu_0>0, \,{\rm when} \, n\rightarrow \infty.
\end{equation}
\end{proof}
We can now obtain global convergence thanks to proposition \ref{compacite solution KP}.
\begin{proposition}\label{convergence L 2}
Let $(p_n)_{n \in \mathbb{N}}$ a sequence which converges to 0. Up to an extraction, there exists a ground state $N_0$ such that 
\begin{equation}
    \partial_1 \theta_{p_n} \rightarrow N_0 \;  in \; Y(\mathbb{R}^2), \quad \mbox{and} \quad N_{p_n} \rightarrow N_0 \; in \; L^2(\mathbb{R}^2),\, \mbox{when} \;n \rightarrow \infty.
\end{equation}
\end{proposition}
\begin{proof}
Using proposition \ref{proposition7.2} there exists a  constant $\mu_0$, such that, up to a subsequence,
\begin{equation}
    E_{KP}(\partial_1 \theta_{p_n}) \rightarrow \mathcal{E}^{KP}_{min}(\mu_0), \quad {\rm and} \quad \int_{\Ra^2}|\partial_1 \theta_{p_n}|^2 \rightarrow \mu_0>0, \,{\rm when} \, n\rightarrow \infty.
\end{equation}
Then using proposition \ref{compacite solution KP}, there exists $(a_n)_{n \in \mathbb{N}}$ and a ground state $\overset{\sim}{N_0}$  with speed $\sigma=\frac{\mu_0^2}{(\mu^*)^2} $, such that
\begin{equation}
    \partial_1 \theta_{p_n}(.-a_n) \underset{n \rightarrow \infty}{\longrightarrow} \overset{\sim}{N_0} \, \text{in} \, Y(\Ra^2) .
\end{equation}
So by \eqref{meilleurconvergence} we have
\begin{equation}
\label{7.14}
    N_{p_n}(.-a_n) \underset{n \rightarrow \infty}{\longrightarrow} \overset{\sim}{N_0} \; {\rm in} \; L^2(\mathbb{R}^2).
\end{equation}
Using proposition \ref{convergencefortelocal} we have 
$$\forall 1<q<\infty,\, \forall k \in\mathbb{N}, \, N_{p_n}(.-a_n) \rightharpoonup N_0 \,\; \text{in} \, \, W^{k,q}\big(\mathbb{R}^2\big), \, \text{when}\; n \rightarrow \infty,
$$
with  $N_0$ a solution of \eqref{SW}. Thus by \eqref{7.14} $N_0=\overset{\sim}{N_0}$, and $\overset{\sim}{N_0}$ is a ground state of speed 1.
We will now prove the convergence of $N_{p_n}$ and $\partial_1 \theta_{p_n}$. Using the continuity of the translation in $L^q(\Ra^2)$ for any $1 \leq q<\infty$, if $a_n\rightarrow a$ then
$$
  \partial_1 \theta_{p_n}(.-a) \underset{n \rightarrow \infty}{\longrightarrow} N_0 \;  {\rm in} \; Y(\mathbb{R}^2), \quad {\rm and} \quad N_{p_n}(.-a) \underset{n \rightarrow \infty}{\longrightarrow} N_0 \; {\rm in} \; L^2(\mathbb{R}^2).
$$
Thus, it is sufficient to prove that $(a_n)_{n \in \mathbb{N}}$ is bounded. By contradiction assume, up to an extraction, that $(a_n)_{n \in \mathbb{N}}$ satisfies
$$a_n \underset{n \rightarrow \infty}{\longrightarrow} \infty  .
$$
Then combining \eqref{7.2} and \eqref{proposition 3.3} there exists $C>0$ not depending on $n$ such that
\begin{equation}
    \int_{B(0,1)}N_{p_n}^2 > 2C,
\end{equation}
thus by \eqref{7.14} we have 
\begin{equation}
    \int_{B(0,1)} |N_0(x+a_n)-N_{p_n}(x)|^2dx \underset{n \rightarrow \infty}{\longrightarrow} 0 .
\end{equation}
Then for  $n$ large enough 
\begin{equation}
    \int_{B(0,1)} |N_0(x+a_n)|^2dx \geq C,
\end{equation}
which is absurd  since for any $f \in L^2$ we have
\begin{equation}
    \int_{B(0,1)} |f(x+a_n)|^2dx \underset{n \rightarrow \infty}{\longrightarrow} 0,
\end{equation}
this concludes the proof of proposition \ref{convergence L 2}.
\end{proof}
\begin{proof}[Proof of theorem \ref{thm de l'article}] Combining proposition \ref{convergence L 2} ($L^2$ convergence), theorem \ref{bornes sobolev} (boundedness in $W^{k,q}$) and an interpolation argument we have
\begin{equation}
    N_{p_n} \underset{n \rightarrow \infty}{\longrightarrow} N_0 \; \textrm{in} \; W^{k,q}(\mathbb{R}^2).
\end{equation}
Since the embedding $Y(\Ra^2) \hookrightarrow L^2(\Ra^2)$ is continuous the same argument gives
\begin{equation}
    \partial_1 \theta_{p_n} \underset{n \rightarrow \infty}{\longrightarrow} N_0 \; \text{in} \; W^{k,q}(\mathbb{R}^2).
\end{equation}
\end{proof}
\appendix

\section{Soliton in the one dimensional case}
\label{appendix diml 1}
We begin with some reminders about nonlinear Schrödinger equations in dimension 1. Traveling wave solutions to \eqref{TWc} are related to the soliton of the generic Korteweg-de Vries equation (see \cite{BGS3}). For traveling waves solution to \eqref{NLS} the transonic limit can lead to solitons of the modified \eqref{KdV} equation or even the generalized (KdV) (see \cite{Chiron}). We show in this section that traveling waves of \eqref{E-K} exhibit similar properties. 
\newline
The existence of traveling waves in dimension one for the Euler-Korteweg equation follows from basic ode arguments that we sketch here.
The Euler-Korteweg system, in dimension one, reads
\begin{equation}
\tag{E-K}
  \left\{
      \begin{aligned}
        &-c\rho'+(\rho u)'=0,\,\, \quad  \qquad \qquad \quad \quad \qquad \quad \quad  (1) \\
        &\ -cu'+ \bigg( \frac{u^2}{2} \bigg)'+g'= \bigg( K(\rho) \rho''+\frac{1}{2}K'(\rho)\rho'^2 \bigg)' \quad (2), \quad x\in \mathbb{R}.
      \end{aligned}
    \right.
\end{equation}
We assume
\begin{equation}
    g(1)=0,\;g'(1)=1 \;, \;  \Gamma=g''(1)+3 \ne 0,\; \gamma=\frac{1}{g''(1)+3}, \; c \in ]0,1[.
\end{equation}
We will study solitons whose limits are
\begin{equation}
    \rho_{\pm \infty}=1;\rho'_{\pm\infty}=0 ,\rho''_{\pm\infty}=0, u_{\pm\infty}=0;u'_{\pm\infty}=0 ,u''_{\pm\infty}=0.
\end{equation}
As in dimension two, we let 
\begin{equation}
    \varepsilon=\sqrt{1-c^2}.
\end{equation}
By integrating (1) on $[x,+\infty[$ we obtain
\begin{equation}
    -c\rho+\rho u =-c,
\end{equation}
so \begin{equation}
\label{substitutuion}
    u=\frac{c(\rho-1)}{\rho}.
\end{equation}
Using \eqref{substitutuion} and (2), we have, by integration on $[x,+\infty[$, 
\begin{equation}
\label{**}
    -c^2 \frac{(\rho-1)}{\rho}+ \frac{c^2(\rho-1)^2}{2\rho^2}+g(\rho)=K(\rho) \rho'' +\frac{1}{2}K'(\rho)\rho'^2.
\end{equation}
Multiplying by $\rho'$ and integrating on $[x,+\infty[$, we have
\begin{equation}
\label{***}
     \frac{1}{2}K(\rho)(\rho')^2=\frac{-c^2}{2\rho }(\rho-1)^2+G(\rho).
\end{equation}
where $G$ is a primitive of $g$ such that $G(1)=0$.
\paragraph{ The case $\Gamma \neq 0$}
We define the function
\begin{equation}
    \frac{-c^2}{2\rho }(\rho-1)^2+G(\rho)=F_{\varepsilon}(\rho),
\end{equation}
and
\begin{equation}
\label{rhom}
\begin{split}
\rho_{m,\varepsilon}=\left\{\begin{array}{cc}
       \sup \left\{ \rho<1, F_{\varepsilon}(\rho)=0 \right\}, \; \text{if} \; \gamma >0.\\
       \inf \left\{ \rho>1, F_{\varepsilon}(\rho)=0 \right\}, \; \text{if} \; \gamma <0.
\end{array}
\right.
\end{split}
\end{equation}
Then we have:
\begin{lemme}\label{convergence rho m}
For $\varepsilon\ll 1$, we have
\begin{equation}
\label{mais ou est donc rho m}
     1-(3\gamma\varepsilon^2+|\gamma| \varepsilon^3)\leq \rho_{m,\varepsilon} \leq 1-(3\gamma\varepsilon^2 - |\gamma|\varepsilon^3).
\end{equation}
\end{lemme}
\begin{proof}
The lemma is a direct consequence of the intermediate value theorem   and the fact that 
\begin{align*}
    F_{\varepsilon}(1-\gamma \alpha \varepsilon^2)=\varepsilon^6\alpha^2\gamma^2(1-\frac{\alpha}{3}+\mathcal{O}(\varepsilon^2 \alpha)).
\end{align*}
\end{proof}
\begin{rem} For $\varepsilon\ll 1$,\newline
\begin{itemize}
    \item 
If $\gamma>0$, \,
$0<\rho_{m,\varepsilon} <1$ , $F_{\varepsilon}(\rho_{m,\varepsilon})=0$ {\rm and} $F_{\varepsilon}(\rho)>0  \quad {\rm for} \quad  \rho \in  ]\rho_{m,\varepsilon},1[.$
\end{itemize}
\begin{itemize}
    \item 
    If $\gamma<0$, \,
$1<\rho_{m,\varepsilon}<2$ ,  $F_{\varepsilon}(\rho_{m,\varepsilon})=0$ 
{\rm and} $F_{\varepsilon}(\rho)>0  \quad {\rm for} \quad  \rho \in  ]1,\rho_{m,\varepsilon}[.$
\end{itemize}
\end{rem} 
To continue our reasoning, we need informations on the derivative.
$$
\gamma F_{\varepsilon}'(\rho_{m,\varepsilon})>0, \quad  (\text{if} \; 1-c\ll 1).
$$
Indeed, we will prove, if $ \gamma >0$, that the soliton decreases from 1 to $ \rho_{m,\varepsilon}$ between $]- \infty,0]$ then increases between $[0, \infty[$. If $ \gamma <0$,  the soliton increases from 1 to $ \rho_{m,\varepsilon}$ between $]- \infty,0]$ then decreases between $[0, \infty[$.
But if $F_{\varepsilon}'( \rho_{m,\varepsilon})=0$ then the Cauchy solution of \eqref{**}  such that $ \rho(0)= \rho_{m,\varepsilon}$,\,$ \rho'(0)=0$ is stationary.
  And in this case there is no soliton that converges to 1.
\begin{lemme}
For $\varepsilon\ll 1$
\begin{equation}
\label{ce n'etait pas une hypothèse :0 }
   \gamma F_{\varepsilon}'(\rho_{m,\varepsilon}) >0.
\end{equation}
\end{lemme}
\begin{proof}
Consequence of lemma \ref{convergence rho m} and elementary computation.
\end{proof}
We can now conclude
\begin{proposition}\label{existence}
Under the conditions $$
g(1) = 0, \;g'(1) = 1,\; \Gamma \ne 0 . $$
For $\varepsilon\ll 1$,
there exists $\rho$ solution of \eqref{**} with $$\rho(0)=\rho_{m,\varepsilon}, \;\rho'(0)=0,$$ and $$u=\frac{c(\rho-1)}{\rho}.$$ 
Then $(\rho,u)$ is a global solution of
\begin{equation}
\tag{E-K}
  \left\{
      \begin{aligned}
        &-c\rho'+(\rho u)'=0, \qquad \qquad \quad \quad \qquad \quad \quad  (1) \\
        &\-cu'+ \bigg( \frac{u^2}{2} \bigg)'+g'= \bigg( K(\rho) \rho''+\frac{1}{2}K'(\rho)\rho'^2 \bigg)' \quad (2), \quad x\in \mathbb{R},
      \end{aligned}
    \right.
\end{equation}
such that
$$
    \rho_{\pm\infty}=1;\rho'_{\pm\infty}=0 ,\rho''_{\pm\infty}=0, u_{\pm\infty}=0;u'_{\pm\infty}=0 ,u''_{\pm\infty}=0.
$$
Moreover 
if $\gamma>0$
\begin{itemize}
    \item 
$\rho$ is increasing on $[0,+\infty[$ and even.
\end{itemize}
If $\gamma<0$
\begin{itemize}
    \item 
$\rho$ is decreasing on $[0,+\infty[$ and even.
\end{itemize}
\end{proposition}
\begin{proof}
Multiplying the equation \eqref{**} by $\rho'$ and integrating on $[0,x]$ we obtain 
\begin{equation}
\label{****}
    \frac{1}{2}K(\rho)(\rho')^2=\frac{-c^2}{2\rho }(\rho-1)^2+G(\rho)=F(\rho).
\end{equation}
We have, by \eqref{ce n'etait pas une hypothèse :0 }, $\rho''(0)=F'(\rho_m) $  has the same sign as $ \gamma$.
So near 0, if $\gamma >0$ $\rho'$ is increasing. The rest of the proof is deduced using basic ode argument and phase portrait analysis (see figure \ref{figure1}, \ref{figure2}).
\end{proof}
We let, as in the two-dimensional case 
\begin{equation}
\label{r definition}
    \rho-1=-\varepsilon^2 \gamma r_{\varepsilon}\bigg( \frac{\varepsilon x}{\sqrt{K(1)}} \bigg).
\end{equation}
Then we have
\begin{proposition}\label{prop convergence rho }
Let
\begin{equation}
    N(x)=\frac{3}{ch^2(\frac{x}{2})},
\end{equation}
the classical soliton to the Korteweg-de-Vries equation \eqref{Kdv}.
Then 
\begin{equation}
\label{convergence r uiformément global}
    r_{\varepsilon} \underset{\varepsilon \rightarrow 0}{\longrightarrow} N , \, {\rm in} \; C^0(\mathbb{R}).
\end{equation}
\end{proposition}
\begin{figure}[h!]
	\begin{minipage}[b]{0.48\linewidth}
		\centering \includegraphics[scale=0.4]{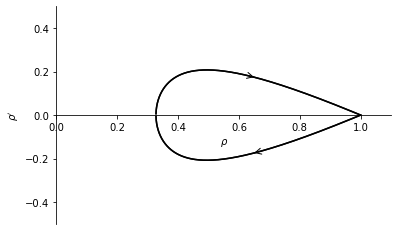}
		\caption{K=1,   $\Gamma >0$,  $\varepsilon=0.82$, \, $G(x)=\frac{(x-1)^2}{2}$.}
		\label{figure1}
	\end{minipage}\hfill
	\begin{minipage}[b]{0.48\linewidth}	
		\centering \includegraphics[scale=0.4]{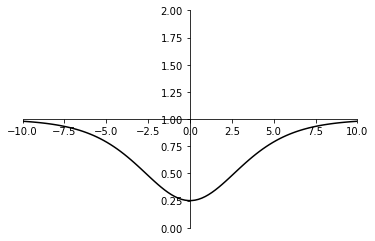}
		\caption{Graph of $\rho(x)$. }
		\label{figure2}
	\end{minipage}\hfill
\end{figure}
First of all we have 
\begin{lemme}
For any $\varepsilon\ll 1$,  $r$ is even,  increasing on $]-\infty,0]$ then decreasing on  $[0,+\infty[$.
Moreover  $$r_{\varepsilon}(0) \underset{\varepsilon \rightarrow 0}{\longrightarrow} 3,$$
and there exists a constant $ C $, not depending on $\varepsilon$, such that
$$
\|r'_{\varepsilon}\|_{\infty} \leq C.
$$
\end{lemme}
\begin{proof}
The variations of r are an immediate consequence of the proposition \ref{existence}. Using \eqref{r definition} and lemma \ref{convergence rho m} we have for $\varepsilon\ll 1$
$$
3-\frac{\varepsilon}{| \gamma |}\leq r_{\varepsilon}(0)\leq 3+\frac{\varepsilon}{| \gamma |}.
$$
Using \eqref{****} we obtain
$$
\frac{\gamma^2}{2K(1)}K(\rho)  r_{\varepsilon}'^2=\frac{r_{\varepsilon}^2\gamma^2}{2}\bigg( 1-r_{\varepsilon}\gamma \frac{1}{3\gamma}\bigg) + O(\varepsilon^2 r_{\varepsilon}^3).
$$
Since $K(1) \ne 0  $ and \; $\rho \underset{\varepsilon \rightarrow 0}{\longrightarrow} 1 \, \text{in} \; C^0(\mathbb{R})  $
 there exists $M$ not depending on  $\varepsilon$ such that 
$$
\|r'_{\varepsilon}\|_{\infty} \leq M,
$$
this ends the proof.
\end{proof}
\begin{proof}[Proof of proposition \ref{prop convergence rho }]
We recall that  
\begin{equation}
    -c^2 \frac{(\rho-1)}{\rho}+ \frac{c^2(\rho-1)^2}{2\rho^2}+g(\rho)=K(\rho) \rho'' +\frac{1}{2}K'(\rho)\rho'^2.
\end{equation}
After lenghty but simple computations, we find that $(r_{\varepsilon},r_{\varepsilon}')$ is a solution of
$$X'=f(t,X,\varepsilon),\; X(0)= \bigg(\begin{array}{c}
r_\varepsilon(0) \\
0 
\end{array} \bigg),$$
with 
\begin{align*}
    f&:\mathbb{R}\times U\times]-\delta,\delta[ \longrightarrow \mathbb{R}^2
    \\
    &f\bigg( t, \bigg(\begin{array}{c}
r \\
r' 
\end{array} \bigg),\varepsilon\bigg)=\bigg(\begin{array}{c}
r' \\
g(r,r',\varepsilon)
\end{array} \bigg),
\end{align*}
where $U =]0,4[ \times ]-M-1,M+1[$,
\begin{align*}
g(r,r',\varepsilon)=r-\frac{r^2}{2}+ 
    \varepsilon^2R_1(\varepsilon^2r) + \varepsilon^2r^3P_1(\varepsilon^2r)
    +\varepsilon^2R_2(\varepsilon^2r)r'^2.
\end{align*}
with 
$P_1,R_1,R_2 \in C^{\infty}(]-1,1[)$.
 We recall that
$$
N(x)=\frac{3}{ch^2(x/2)}.
$$
Moreover $N$ is a solution of
$$
 X'=f(t,X,0), \; X(0)= \bigg(\begin{array}{c}
3 \\
0 
\end{array} \bigg).
$$
Then as $r_{\varepsilon}$ is a solution of
$$X'=f(t,X,\varepsilon),\; X(0)= \bigg(\begin{array}{c}
r_\varepsilon(0) \\
0 
\end{array} \bigg), $$
with $r_\varepsilon(0) \rightarrow 3,$ we obtain, using the Cauchy-Lipschitz theorem with parameter, that for any compact set  $[a,b]$
\begin{equation}
   r_{\varepsilon} \underset{\varepsilon \rightarrow 0}{\longrightarrow} N , \, \text{in} \; C^0([a,b]),
\end{equation}
and
\begin{equation}
\label{dérivé convergence r local}
   r'_{\varepsilon} \underset{\varepsilon \rightarrow 0}{\longrightarrow} N' , \, \text{in} \; C^0([a,b]).
\end{equation}
We have using proposition \ref{existence} and \eqref{r definition} that the functions $N,r_{\varepsilon}$ are increasing on $]-\infty,0]$, decreasing on $[0,\infty[ $  and converge to 0 in $+\infty$ and $-\infty$. Thus we obtain 
$$
 r_{\varepsilon} \underset{\varepsilon \rightarrow 0}{\longrightarrow} N , \, \text{in} \; C^0(\mathbb{R}).
$$
which concludes the proof.
\end{proof}
The following proposition ends the proof of proposition  \ref{dim 1 convergence gamma pas 0}
\begin{proposition}
\label{convergence uniforme proposition finale r}
We have for all $k\in \mathbb{N}$
$$
\|r_{\varepsilon}^{(k)} -N^{(k)}\|_{L^{\infty}(\mathbb{R})} \underset{\varepsilon \rightarrow 0}{\longrightarrow} 0.
$$
\end{proposition}
\begin{proof}
We already know the result for $k=0$, we prove it for $k=1$. We have by \eqref{dérivé convergence r local} uniform convergence on any compact.
As $r_{\varepsilon}$ converges uniformly to $N$   and $N(x) \underset{|x| \rightarrow \infty}{\longrightarrow} 0$ then $\underset{\varepsilon'\leq \varepsilon}{sup}\,\underset{M\leq |x|}{sup}\, r_{\varepsilon'}(x) \underset{M \rightarrow \infty}{\longrightarrow} 0$. Since we have
\begin{equation}
    \frac{\gamma^2}{2K(1)}K(\rho)  r_{\varepsilon}^{'2}=\frac{r_{\varepsilon}^2\gamma^2}{2}\bigg( 1-r_{\varepsilon}\gamma \frac{1}{3\gamma}\bigg) + O(\varepsilon^2 r_{\varepsilon}^3). 
\end{equation}
we obtain that $\underset{\varepsilon'\leq \varepsilon}{sup}\,\underset{M\leq |x|}{sup}\,r_{\varepsilon'}^{'}\underset{|x| \rightarrow \infty}{\longrightarrow} 0$. Moreover $N'(x) \underset{|x| \rightarrow \infty}{\longrightarrow} 0$ so we obtain
$$ \forall \eta >0, \exists \varepsilon', \exists M, \big( |x|>M, \varepsilon <\varepsilon' \implies |r'_{\varepsilon}(x)-N(x)| \leq |r'_{\varepsilon}(x)|+|N(x)| \leq \eta \big),$$
combining with \eqref{convergence r uiformément global} we have 
\begin{equation}
\label{convergence r' global}
    \|r_{\varepsilon}^{(1)} -N^{(1)}\|_{L^{\infty}(\mathbb{R})} \underset{\varepsilon \rightarrow 0}{\longrightarrow} 0.
\end{equation}
The result for higher order follows by use of the ODE and a simple induction argument.
\end{proof}
\begin{rem}
Using \eqref{substitutuion} and letting
\begin{equation}
    u(x)=-\varepsilon^2\gamma v_{\varepsilon}\bigg( \frac{\varepsilon x}{\sqrt{K(1)}}\bigg),
\end{equation}
we obtain
$$
v_{\varepsilon}(x)=\frac{\sqrt{1-\varepsilon^2}}{1 -\varepsilon^2\gamma r_{\varepsilon}(x)}r_{\varepsilon}(x),
$$
so by the proposition \ref{convergence uniforme proposition finale r} we have that 
$$\|r_{\varepsilon} -N\|_{L^{\infty}(\mathbb{R})} \underset{\varepsilon \rightarrow 0}{\longrightarrow} 0.$$
On the other hand by using Leibniz formula one has, for all $n \in \mathbb{N}$\begin{align*}
    v_{\varepsilon}(x)= \sum_{k=0}^{n-1}\bigg(\frac{\sqrt{1-\varepsilon^2}}{1 -\varepsilon^2\gamma r_{\varepsilon}(x)}\bigg)^{(k)} r_{\varepsilon}^{(n-k)}(x)+\bigg(\frac{\sqrt{1-\varepsilon^2}}{1 -\varepsilon^2\gamma r_{\varepsilon}(x)}\bigg) r_{\varepsilon}^{(n)}(x),
\end{align*}
and therefore by proposition \eqref{convergence uniforme proposition finale r} we have 
for all $n \in \mathbb{N}$
\begin{equation}
    \|v^{(n)}_{\varepsilon} -N^{(n)}\|_{L^{\infty}(\mathbb{R})} \underset{\varepsilon \rightarrow 0}{\longrightarrow} 0,
\end{equation}
a result similar to the one obtained in dimension 2.
\end{rem}
\paragraph{The case $\Gamma=0$}
The proof of proposition \ref{proposition gamma = 0 } is similar to what was done earlier (see figure \ref{figure12} for a phase portrait). let us give a heuristic argument. Substituting
$$
\rho=1+\varepsilon r\bigg(\frac{\varepsilon x}{\sqrt{K(1)}}\bigg),
$$ into \eqref{**} we find 
$$
\bigg(\frac{g'''(1)-12}{6}\bigg)r^3+r=r''+O(\varepsilon).
$$
Then differentiating this equation we have 
$$
\frac{1}{2}w^2w'+w'=w'''.
$$
where $$w=\frac{1}{\sqrt{g'''(1)-12}}.$$
For a similar (and more general) result  in the case of the non-linear Schrödinger  equation, see \cite{Chiron}. 
\begin{figure}[h!]
	\begin{minipage}[b]{0.8\linewidth}
		\centering \includegraphics[scale=0.6]{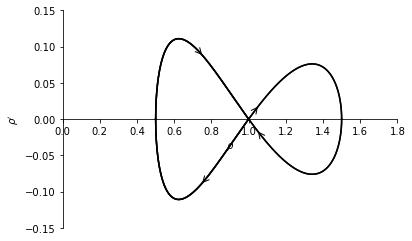}
    \caption{K=1,   $\Gamma =0$,  $\varepsilon$=0.5, \, $G(x)=\frac{(x-1)^2}{2}-\frac{(x-1)^3}{2}$.}\label{figure12}
	\end{minipage}\hfill
\end{figure}
\newline\newline\newline\newline\newline\newline\newline\newline\newline\newline\newline
\bibliography{biblio}
\bibliographystyle{plain}
\end{document}